\documentclass[12pt]{article}

\usepackage{fullpage,amsfonts,amsmath,amsthm,amssymb,mathtools,mathrsfs,graphicx,upgreek}
\usepackage{verbatim,url}
\usepackage{graphicx,hyperref,enumitem}
\usepackage[noadjust,sort]{cite}
\usepackage{tocloft}

\usepackage[normalem]{ulem}

   \usepackage[top=2.54cm, bottom=2.54cm, left=2.54 cm, right=2.54 cm]{geometry}
\newcommand{\lab}[1]{\label{#1}}                % hides labels

 \usepackage[usenames,dvipsnames]{color}
\newcommand{\remove}[1]{}
\newcommand\eqn[1]{(\ref{#1})}

\newcommand{\be}{\begin{equation}}
\newcommand{\bel}[1]{\begin{equation}\lab{#1}\ }
\newcommand{\ee}{\end{equation}}
\newcommand{\bea}{\begin{eqnarray}}
\newcommand{\eea}{\end{eqnarray}}
\newcommand{\bean}{\begin{eqnarray*}}
\newcommand{\eean}{\end{eqnarray*}}

\newtheorem{thm}{Theorem}%[section]
\newtheorem{cor}[thm]{Corollary}
\newtheorem{con}[thm]{Conjecture}
\newtheorem{lemma}[thm]{Lemma}

\newtheorem{claim}[thm]{Claim}

\newtheorem{example}[thm]{Example}

\newcommand{\ceil}[1]{\left\lceil#1\right\rceil}

\newcommand{\brk}[1]{{\left(#1\right)}}

\newcommand{\ind}[1]{1_{\{#1\}}}

\def\id{\text{id}}
\def\aut{\text{Aut}}

\newcommand{\bm}[1]{{\pmb #1}}
\def\coloneqq{{:=}}

\def\G{{\mathcal G}}

\def\P{{\mathcal P}}

\newcommand{\Aut}{\text{Aut}}

\def\ex{{\mathbb E}}
\def\pr{{\mathbb P}}

\def\bfd{{\bf d}}

\def\bfp{{\bf p}}

\def\bfu{{\bf u}}

\def\bfU{{\bf U}}

\def\eps{\epsilon}

\date{}

\title{The threshold of symmetry in random graphs with specified degree sequences}
\author{Lochlan Brick \\
Monash University\\
lochiebrick@gmail.com
\and Pu Gao\thanks{Research supported by ARC DE170100716 and NSERC. The major part of the research was conducted when the author was affiliated with Monash University.}\\
University of Waterloo\\
pu.gao@uwaterloo.ca\\
 \and Angus Southwell\\
 Monash University\\
 angus.southwell@monash.edu
}
\begin{document}
\maketitle

\begin{abstract}

We give sufficient conditions under which a random graph with a specified degree sequence is symmetric or asymmetric. In the case of bounded degree sequences, our characterisation captures the phase transition of the symmetry of the random graphs. This phase transition coincides with that of the graph connectivity. However, when the maximum degree is a growing function as the number of vertices tends to infinity, our results suggest that these two thresholds do not coincide any more.  
\end{abstract}

\section{Introduction}

The study of random graphs was initiated by Erd\H{o}s and R\'{e}nyi~\cite{Erdos59,Erdos60}, and by Gilbert~\cite{Gilbert59} in the mid 20th century. A random graph is a probability measure over a set of $n$-vertex graphs. For instance, Erd\H{o}s and R\'{e}nyi introduced the random binomial graph $\G(n,p)$ in which each edge between any pair of vertices appears independently with probability $p$. Since the seminal work~\cite{Erdos59,Erdos60} by  Erd\H{o}s and R\'{e}nyi, various random graph models have been introduced. For any graph property $A_n$, we say $A_n$ holds asymptotically almost surely (a.a.s.) in a random graph $\G_n$ if $\pr_{\G_n}(A_n)$ goes to 1 as $n\to\infty$. Graph symmetry is among the earliest properties that were investigated by Erd\H{o}s and R\'{e}nyi~\cite{Erdos63} for random graphs $\G(n,p)$. A graph is said to be {\em asymmetric} if the automorphism group of $G$ contains only the identity permutation. Otherwise, the graph is said to be {\em symmetric}. Erd\H{o}s and R\'{e}nyi~\cite{Erdos63} showed that if $\min \{p,1-p\}\ge (1+\eps)\ln n/n$, then $\G(n,p)$ is a.a.s.\ asymmetric. If $\min \{p,1-p\}\le (1-\eps)\ln n/n$ then either $\G(n,p)$ or its complement contains many isolated vertices, resulting in symmetry. Note that $p=\ln n/n$ is the threshold for $\G(n,p)$ being connected. Thus, the phase transition of graph symmetry coincides with that of graph connectivity in $\G(n,p)$ until $p$ gets very close to 1.

The fact that the symmetry of $\G(n,p)$ is caused by isolated vertices immediately motivates the study of symmetry of graphs with a specified degree sequence. Let ${\bf d}=(d_1,\ldots, d_n)$ be a degree sequence where $M_1=\sum_{i=1}^n d_i$ is necessarily even. Let $\G({\bfd})$ denote the uniformly random graph with degree sequence $\bfd$. We use notation $\G(n,d)$ in the case where all components of ${\bfd}$ are equal to $d$. If $d\in\{0,1,2,n-3,n-2,n-1\}$ then it is easy to see that $\G(n,d)$ is a.a.s.\ symmetric due to its simple structure. The symmetry of $\G(n,d)$ for non-trivial values of $d$ was studied by Bollob\'{a}s~\cite{Bollobas} for bounded $d\ge 3$, and was studied by McKay and Wormald~\cite{McKay} for $d=o(\sqrt{n})$. Finally, Kim, Sudakov, and Vu~\cite{Kim02} proved that a random $d$-regular graph is a.a.s.\ asymmetric for all $3\le d\le n-4$. In the paper~\cite{McKay} by McKay and Wormald, the symmetry of $\G({\bfd})$ for more general degree sequences $\bfd$ was studied. In particular, they proved that if $\bfd$ is a ``moderate'' degree sequence (e.g.\ the maximum degree is not too large) where the minimum degree is at least 3, then a.a.s.\ $\G(n,\bfd)$ is asymmetric. In the case where the minimum degree is 1 or 2, they also gave sufficient conditions under which $\G(\bfd)$ is a.a.s.\ asymmetric. 
Let
\[
n_j=\sum_{i=1}^n 1_{\{d_i=j\}} \quad \mbox{and}\quad n_{\ge j}=\sum_{i=1}^n 1_{\{d_i\ge j\}}.
\]
 In the case of bounded degree sequences, McKay and Wormald's result~\cite{McKay} implies the following.
\begin{thm}[\cite{McKay}]\label{thm:mw}
Fix $\Delta>0$ and assume that $\bfd$ is a degree sequence where $d_i\le \Delta$ for all $1\le i\le n$. If $
n_1 =O(n^{1/3})$ and $n_2 =O(n^{2/3})$,
then a.a.s.\ $\G({\bfd})$ is asymmetric.
\end{thm}

In this paper, we strengthen the sufficient conditions in~\cite{McKay}.  Recall that for a graph property $Q$, we say $\G(n,p)$ has a threshold $p_0(n)$ with respect to $Q$ if 
\[
\lim_{n\to\infty}\pr(\G(n,p)\in Q)=\left\{
\begin{array}{ll}
0 & \mbox{if $p\ll  p_0(n)$} \\
1 & \mbox{if $p\gg p_0(n)$.}
\end{array}
\right.
\]

First we generalise the concept of threshold to $\G(\bfd)$. However, $\G(\bfd)$ is a much more complicated probability space to handle, and typically for the purpose of analysis we need to impose some conditions on $\bfd$ such as an upper bound on $\max_{1\le j\le n} d_j$. Thus, instead of considering all graphical degree sequences, we will consider a sequence of degree sequences $({\mathcal D}_n)_{n\in {\mathcal I}}$ indexed by $n\in {\mathcal I}$, where ${\mathcal I}$ is an infinite subset of ${\mathbb Z}_{\ge 1}$ (e.g.\ ${\mathcal I}={\mathbb Z}_{\ge 1}$),  and all degree sequences in ${\mathcal D}_n$ are of form $\bfd=(d_1,\ldots, d_n)$. Given a sequence $({\mathcal D}_n)_{n\in {\mathcal I}}$, we say graph property $Q$ has a sharp threshold, if there exist functions $f_n: {\mathbb R}^n \to {\mathbb R}$ such that for all $\bfd\in {\mathcal D}_n$,
\[
\lim_{n\to\infty}\pr(\G(\bfd)\in Q)=\left\{
\begin{array}{ll}
0 & \mbox{if $f_n(\bfd)\ll 1$} \\
1 & \mbox{if $f_n(\bfd)\gg 1$.}
\end{array}
\right.
\]
For the study of the symmetry of $\G(\bfd)$ we may assume without loss of generality that all components of $\bfd$ are at least 1. Thus, all degree sequences $\bfd$ considered in this paper have no zero components. In the case of bounded degree sequences, our result captures the threshold of graph symmetry of $\G(\bfd)$; this is given in the following theorem.
\begin{thm}\label{thm:bounded}
Fix $\Delta>0$ and assume that $\bfd$ is a degree sequence where $1\le d_i\le \Delta$ for all $1\le i\le n$.
\begin{enumerate}
\item[(a)] If $n_1=o(n^{1/2})$ and $n_2=o(n)$
then a.a.s.\ $\G({\bfd})$ is asymmetric. 
\item[(b)] If $n_1 =\omega(n^{1/2})$ then a.a.s.\ $\G(\bfd)$ is symmetric. 
\item[(c)] If there is a constant $c > 0$ such that
\[
n_1 > cn^{1/2} \quad\mbox{or} \quad n_2 > cn 
\]
then there is $\delta=\delta(c)>0$ such that for all sufficiently large $n$
\[
\pr(\G(\bfd) \mbox{ symmetric}) > \delta.
\]
\end{enumerate}
\end{thm}
The symmetry threshold coincides exactly with that of graph connectivity for random graphs with bounded degree sequences \cite{Federico17}.

Given $\bfd=(d_1,\ldots,d_n)$ and a positive integer $i$, let $M_i=\sum_{j=1}^n (d_j)_i$, where $(x)_i=\prod_{j=0}^{i-1}(x-j)$ denotes the $i$-th falling factorial of $x$. Let 
\[
\Delta=\max\{d_j:\ 1\le j\le n\}, \quad d=\frac{M_1}{n}.
\]
Theorem~\ref{thm:bounded} follows as a corollary of the following results regarding the symmetry of $\G(\bfd)$ for more relaxed degree sequences.
\begin{thm}\label{thm:sub}

Suppose there are constants $R_1>0$, $R_{2}>0$ and $0<\eps<1$ such that degree sequence $\bfd$ satisfies the following conditions:
\begin{itemize}
\item $\frac{\Delta^2}{d}=o(n^{\frac{1}{6}-\frac{1}{2R_1}-\frac{1}{R_2}})$, \quad $\frac{\Delta^2}{d}=o\left(\frac{n^{1/4}}{n_1^{1/2}}\right)$; \quad $\frac{\Delta^2}{d}=o\left(\frac{n^{\alpha_2/2}}{n_2^{\alpha_2/2}}\right)$;
\item $\left(\frac{n_i}{n}\right)^{\alpha_i(1-\eps)} \left(\frac{\Delta^2}{d}\right)^{2-\eps}=o(1),\ \mbox{for $i\in \left\{1,2\right\}$.} $
\end{itemize}
where
$\alpha_2=1/(R_2+4)$, and $\alpha_1=(1-\alpha_2)/(R_1+4)$.
Then $\G(\bfd)$ is a.a.s.\ asymmetric.

%Fix $0<c<1/56$. Assume $\Delta^2\le M_1/n^{1-c}$. If the degree sequence $\bfd$ satisfies
%\begin{itemize}
%	\item
%	$n_1^{1/2} = o\left(\frac{M_1}{\Delta^2 n^{3/4}}\right)$,
%	\item
%	$ n_j^{\alpha_j/2}= o\left(\frac{M_1}{ \Delta^2 n^{1 - \alpha_j/2}}\right)$ for $j=1,2$,
	%\item
	%$\Delta^2 = o\left(\frac{M_1}{n^{1-C_0}}\right)$ where $C_0 = \frac{25}{24 + 8\delta} - \frac{7}{8}$,
%\end{itemize}
%for
%$\alpha_2=1/(R_2+4)$, and $\alpha_1=(1-\alpha_2)/(R_1+4)$,
%where $R_1,R_2>0$ are some real numbers satisfying
%\[
% \frac{1}{6R_1}+\frac{1}{3R_2}=\frac{25}{8c+7}-\frac{7}{2},
%\]
%then $\G(\bfd)$ is a.a.s.\ asymmetric.
\end{thm}

In the next theorem, we give sufficient conditions under which $\G(\bfd)$ has a non-zero probability of being symmetric.

\begin{thm}\label{thm:super}
Assume $\Delta^2=o(M_1)$. % $M_1M_3=o(n_1M_2^2)$ and $M_4=o(M_2^2)$.
\begin{enumerate}
\item[(a)]If
$n_1 =\omega\brk{M_1/\sqrt{M_2}}$ or $n_2=\omega\brk{\sqrt{M_1^3/M_2}}$ then a.a.s.\ $\G({\bfd})$ is symmetric.
\item[(b)] If there is a constant $c > 0$ such that
\[
n_1 > c M_1/\sqrt{M_2},\quad\mbox{or}\quad n_2 > c \sqrt{M_1^3/M_2},
\]
then there exists $\delta=\delta(c)>0$ such that for all sufficiently large $n$,
\[
\pr(\G(\bfd)\mbox{ symmetric})>\delta.
\]
\end{enumerate}
\end{thm}

\noindent {\em Remark.} 
(a) In the case $\Delta\to\infty$ slowly, there is a gap for $\bfd$ between Theorems~\ref{thm:sub} and~\ref{thm:super} inside which we cannot determine if $\G(\bfd)$ is a.a.s.\ asymmetric. This was caused by a non-tight bound on the probability of the appearance of a subgraph in $\G(\bfd)$. See Lemma~\ref{lemma:edge-prob-bound} below. We conjecture that conditions in Theorem~\ref{thm:super} capture the critical point of the phase transition, whereas the conditions in Theorem~\ref{thm:sub} are sufficient but not necessary. 

(b) If $M_1=\Omega(n)$, $M_2=O(M_1)$, then the connectivity transition occurs at the point where isolated edges or isolated triangles start to appear with a  probability away from 0. We conjecture that this is still the case when $M_2/M_1$ is a sufficiently slowly growing function of $n$ and $\max_i d_i$ is not too large. Theorem~\ref{thm:super} implies that there are degree sequences where $\G({\bfd})$ is a.a.s.\ connected but is symmetric with a probability away from 0, in contrast to the bounded degree case. We give a concrete example below. The proof is presented in Appendix.

\begin{example}\label{example-gap}
Assume $\bfd$ is a degree sequence with $n_1$ components equal to 1, and $n-n_1$ components equal to $\ceil{\log n}$. Assume $\sqrt{n}\ll n_1\ll \sqrt{n\log n}$. Then $\G(\bfd)$ is a.a.s. symmetric and connected.
\end{example}

{\em Proof of Theorem~\ref{thm:bounded}. } Part (a) follows from Theorem~\ref{thm:sub} by taking $R_1=R_2=10$ and $\eps=1/2$. Parts (b,c) follow by Theorem~\ref{thm:super}.\qed
\medskip

%\noindent {\bf{\em Other related work}}
%\jca{Lochie's thesis can be reused.}
%\begin{itemize}
%\item graph and random graph isomorphism;
%\item isomorphism test in graphs and random graphs;
%\item enumeration of unlabelled graphs;
%\item measures on asymmetry.
%\end{itemize}

{\bf Other related work.} Erd\H{o}s and R\'enyi \cite{Erdos63} defined the degree of asymmetry $A(G)$ of a graph $G$ to be the minimum number of edges that need to be added and/or removed from $G$ in order to produce a symmetric graph.  They show that if $G$ has $n$ vertices, then the maximum value that $A(G)$ can take is $(n-1)/2$, but also that for any $\epsilon > 0$,
\[
\lim_{n\rightarrow \infty} \pr \left( A(G(n,1/2)) \leq \frac{n(1-\epsilon)}{2} \right) = 0.
\]
This result shows  for large $n$, a typical graph in $G(n,1/2)$ is as asymmetric as possible.

Wright \cite{Wright71} proved an analogous result for the random graph $\G_{n,M}$, which is a graph chosen uniformly at random (u.a.r.) from the set of all $n$ vertex graphs containing $M$ edges.  
In particular, Wright proved necessary and sufficient conditions under which the number of labelled $n$-vertex graphs with $M$ edges is asymptotic to $n!$ times the number of unlabelled $n$-vertex graphs with $M$ edges.  
This is equivalent to proving that $\ex |\Aut(\G_{n,M})|= 1 + o(1)$, which by Markov's inequality implies that $\G_{n,M}$ a.a.s. has no non-trivial automorphisms.
If we define $\mu \coloneqq M / n - \ln(n) / 2$, then Wright's result shows that $\ex |\Aut(\G_{n,M})| = 1 + o(1)$ if and only if $\mu \rightarrow \infty$ as $n \rightarrow \infty$.
This result built on previous work on the same problem by Poly\'a and Oberschelp, while a second proof was later found by Bollob\'as \cite{Bollobas}.  Given that the automorphism group of a graph is the same as that of its complement, when studying $\Aut(\G_{n,M})$ it suffices to consider the case where $M \leq {n \choose 2}/2$.  Later, Wright \cite{Wright74} also proved that if $\mu \rightarrow -\infty$ then $\G_{n,M}$ is a.a.s. symmetric, and moreover that if $\mu \leq 0$ then a.a.s.\ an \emph{unlabelled} graph chosen u.a.r.\ from the set of all unlabelled graphs on $n$ vertices with $M$ edges has an unbounded number of automorphisms.%for any fixed integer $R$, an \emph{unlabelled} graph chosen u.a.r. from the set of all unlabelled graphs on $n$ vertices with $M$ edges a.a.s. has more than $R$ automorphisms %(since we have not defined automorphisms of unlabelled graphs, here we mean that any labelled graph in the corresponding isomorphism class has more than $R$ automorphisms).

Linial \cite{Linial16} observed that by applying the Chernoff bound to show that the number of edges in $\G(n,p)$ is concentrated, Wright's results regarding $\G_{n,M}$ may be extended to $\G(n,p)$ for appropriate values of $p \leq 1/2$.  They then went on to show that for a much lower range of $p$ (below $n^{-1/2})$, the 2-core of $\G(n,p)$ (i.e. the maximum subgraph with minimum degree at least 2) a.a.s. has no non-trivial automorphisms.  

Although the results of Erd\H{o}s, R\'enyi, and Wright are particularly powerful when considering general classes of graphs, they are less useful for studying symmetries of sparse graphs. This is because, if we consider $\G(n,p)$ and $\G_{n,M}$ with small $p$ and small $M$ respectively, these graphs will typically contain isolated vertices.  Isolated vertices make the question of symmetry trivial --- since swapping any two isolated vertices gives a non-trivial automorphism --- and it is therefore difficult to use these models to examine the effect of having many vertices of degree 1 and 2 in particular.  On the other hand, focusing on random graphs with specific degree sequences facilitates the study of sparse graphs with no isolated vertices.

\section{Preliminary and notation}
\label{sec-prelim}

Given a group $G$ and a set $X$, a \emph{group action} of $G$ on $X$ is a map $\varphi \colon G \times X \to X$ (we will write $x^g \coloneqq \varphi(g,x)$) for which $x^{id} = x$ and $x^{gh} = (x^g)^h$, where ${id}$ is the identity in $G$.  The \emph{orbit} of $x \in X$ under this action is the set of all $x' \in X$ for which $x' = x^g$ for some $g \in G$, and the set of orbits forms a partition of the set $X$.
	
Given $\bm{d}$, let $V_i \subseteq[n]$ be the set of vertices having degree $i$ in $\G(\bm{d})$.  Define $S_n^{\bm{d}}$ to be the largest subgroup of the permutation group $S_n$ such that for every $i$, the set $V_i$ is mapped to itself by every element of $S_n^{\bm{d}}$.  Note that $\aut(\G(\bm{d}))$ is always a subgroup of $S_n^{\bm{d}}$, since an automorphism must preserve the degree of each vertex. Define the action of $S_n^{\bm{d}}$ on the set of graphs $\G(\bm{d})$ in a natural way.  
Firstly, given a permutation $\sigma \in S_n$, let $\sigma^*$ be the corresponding permutation of the collection of all 2-subsets of $[n]$, where $\sigma^*(\{i,j\}) \coloneqq \{\sigma(i), \sigma(j)\}$.   
Given $G = (V,E) \in \G(\bm{d})$ and $\sigma \in S_n^{\bm{d}}$, define
\[
	G^\sigma \coloneqq (V, \sigma^*(E)), \quad\text{where}\quad \sigma^*(E) \coloneqq \{\sigma^*(e) :\ e \in E\}.
\]
Note that $S_n^{\bm{d}}$ is the largest subgroup of $S_n$ which guarantees that $G^\sigma \in \G(\bm{d})$ for all $\sigma \in S_n^{\bm{d}}$.
	% Given $G \in \scrG_{n, \bm{d}}$, the \emph{orbit} of $G$ under the action of $S_n^{\bm{d}}$  is the set of graphs $G' \in \scrG_{n, \bm{d}}$ for which there exists $\sigma \in S_n^{\bm{d}}$ such that $G' = G^\sigma$.  
	% The set of orbits in $\scrG_{n, \bm{d}}$ define a partition of $\scrG_{n, \bm{d}}$.
It is easy to check that $\sigma$ is an isomorphism from $G$ to $G'$ if and only if $G' = G^\sigma$, and therefore the isomorphism classes in $\G(\bm{d})$ are exactly the set of orbits in $\G(\bm{d})$ under the action of $S_n^{\bm{d}}$.
	
Now, suppose that $\sigma$ is an automorphism of $G$.  This means that for every edge $e \in E(G)$, we know that $\sigma^*(e) \in E(G)$.  If we think of $\text{Aut}(G)$ acting on $E(K_n)$ by $e^\sigma \coloneqq \sigma^*(e)$, then it follows that $E(G)$ must be a union of orbits in $E(K_n)$.  In other words, if $E(G)$ contains an edge $e$ then it must also contain every edge in the orbit of $e$ in $E(K_n)$.  We will make use of this in the proof of Theorem \ref{thm:sub}.

%\begin{itemize}
%\item $\aut(G)$ and $S_n$;
%\item Let $V_i=\{j\in[n]:\ d_j=i\}$;
%\item $\supp{\sigma}$ for $\sigma\in S_n$.
%\item $\sigma^*$: permutation on $E(K_n)$ lifted from $\sigma\in S_n$.
%\end{itemize}

Throughout the paper we use standard Landau asymptotic notations. For two sequences of real numbers $a_n$ and $b_n$, we write $a_n=O(b_n)$ if there is a constant $C>0$ such that $|a_n|<C|b_n|$ for all $n\ge 1$. We use $a_n=o(b_n)$ if $\lim_{n\to \infty} a_n/b_n=0$. We write $a_n=\omega(b_n)$ if $b_n=o(a_n)$ and $a_n>0$ for all sufficiently large $n$, and we denote $a_n=\Omega(b_n)$ if $b_n=O(a_n)$, and $a_n>0$ for all sufficiently large $n$. All asymptotics in the paper refer to $n\to\infty$.  

The proof of Theorem~\ref{thm:super} uses the standard second moment method. When conditions of Theorem~\ref{thm:super} are satisfied, we prove that with a non-vanishing probability, either a ``cherry'' (i.e.\ two degree 1 vertices adjacent to a third vertex), or a ``pendant triangle'' (i.e.\ a triangle containing two vertices with degree equal to 2) appears, resulting in graph symmetry. 

To prove Theorem~\ref{thm:sub}, we will prove 
\begin{equation}
\pr\Big(\exists \pi\in S_n\setminus\{\id\}:\ \pi\in \aut(\G(\bfd))\Big)=o(1). \label{prob}
\end{equation}
We will partition $ S_n\setminus\{\id\}$ into two sets $ S_n\setminus\{\id\}=S_1\cup S_2$, according to the ratio between number of vertices of degree at most 2 that are moved by $\pi$ and the number of vertices of degree at least 3 that are moved by $\pi$. See their definitions in~\eqn{S1} and~\eqn{S2}.  Then, we will prove
\begin{equation}
\ex\brk{\sum_{\pi\in S_2} \ind{\pi\in \aut(\G(\bfd))}}=o(1),\label{exp}
%\sum_{\pi\in S_2} \pr\Big(\pi\in \aut(\G(\bfd))\Big)=o(1), \label{exp}
\end{equation}
by extending the proof of McKay and Wormald~\cite{McKay}. Finally by investigating structures of $\G(\bfd)$ using probabilistic techniques, we bound the probability of having any $\pi\in S_1$ such that $\pi\in\aut(\G(\bfd))$:
\begin{equation}
\pr\Big(\exists \pi\in S_1:\ \pi\in \aut(\G(\bfd))\Big)=o(1). \label{prob2}
\end{equation}
% bound the probability that any permutation $\pi\in S_n\setminus\{\id\}$ is in $\aut(G)$. In order to prove
 Combining~\eqn{exp} and~\eqn{prob2} yields~\eqn{prob}.

\section{Asymmetry: proof of Theorem~\ref{thm:sub}}
\label{sec:sub}

Given $\sigma \in S_n$ and a graph $G$ on vertex set $[n]$, let $H_{\sigma}(G)$ be the spanning subgraph of $G$ containing precisely the set of edges which have at least one end moved by $\sigma$.  Define $\mathscr{H}_\sigma$ to be the set of all spanning subgraphs of $K_n$ for which $\sigma$ is an automorphism and in which every edge has one end moved by $\sigma$.  Clearly $\sigma\in \aut(G)$ if and only if $\sigma\in\aut(H_\sigma(G))$, which is true if and only if $H_\sigma(G) \in \mathscr{H}_\sigma$.  The probability that $\G(\bfd)$ has a non-trivial automorphism is therefore equal to
\begin{equation}\label{eq:prob-union-over-perms}
\pr\left( \bigcup_{\sigma \in S_n \setminus \{id\}} \{H_\sigma(\G(\bfd)) \in \mathscr{H}_\sigma\}\right).
\end{equation}

We now define a set of parameters, given an arbitrary choice of $\sigma \in S_n$ and $H \in \mathscr{H}_\sigma$. 
Let 
\[
V_i=\{j\in[n]: d_j=i\},\quad V_{\ge 3}=\cup_{j\ge 3} V_j.
\]
 Given $\sigma$ and $H$, let $A_i = V_i \cap \text{supp}(\sigma)$ for $i \in \{1,2\}$ and let $A_{\geq 3} = V_{\geq 3} \cap \text{supp}(\sigma)$, where $\text{supp}(\sigma)=\{i\in[n]: i\neq \sigma(i)\}$.  Let $H_0$ be the subgraph of $H$ induced by $A_1 \cup A_2 \cup A_{\geq 3}$, and note that at least one end of every edge in $H$ must lie in $V(H_0)$.  Define $\sigma^*$ to be the permutation induced by $\sigma$ of the edges of the complete graph $K_n$. Define the following:
\begin{align*}
a_j &= |A_i|, i \in \{1,2\}\\% \text{the number of vertices in $V_j$ moved by $\sigma$, for $j \in \{1, 2\}$},\\
a_{\geq 3} &= |A_{\geq 3}|,\\% \text{the number of vertices in $V_{\geq 3}$ moved by $\sigma$},\\
\ell &= \sum_{i=1}^n d_i \ind{i\in A_{\ge 3}},\\ %\text{the sum of the degrees (in $\G(\bfd)$) of the vertices in $A_{\geq 3}$},\\
s_i &= \text{the number of $i$-cycles of $\sigma$},\\
k &= \text{the number of edges in $H_0$},\\
e_1 &= \text{the number of edges in $H_0$ which are fixed by $\sigma^*$},\\
m &= k - f, \text{where $f$ is the number of orbits of $\sigma^*$ which are completely filled by edges of $H_0$}.
\end{align*}
Note that we will only use $s_i$ for $i \leq 6$, and so whenever we write $\prod_i s_i$ the range $1 \leq i \leq 6$ will be implied. 

For each $H \in \mathscr{H}_\sigma$, let $\bm{p_\sigma}(H) = (a_1, a_2, a_{\geq 3}, \ell, \dots)$ be the vector containing the values for each of the above parameters.  Then we partition the set of graphs $\mathscr{H}_\sigma$ according to the value of this parameter vector. Given a possible parameter vector $\bm{p}$, define
\[
\mathscr{H}_{\sigma, \bm{p}} = \{ H \in \mathscr{H}_\sigma : \bm{p_\sigma}(H) = \bm{p}\}.
\]
We can then rewrite \eqref{eq:prob-union-over-perms} as
\begin{equation}
\pr\left( \bigcup_{\bm{p}} \bigcup_{\sigma \in S_n \setminus \{id\}} \{H_\sigma(\G(\bfd)) \in \mathscr{H}_{\sigma, \bm{p}}\}\right),\label{sum}
\end{equation}
where the outer union is over all possible parameter vectors $\bm{p}$.

Let $R_1,R_2>0$ be real numbers  chosen to satisfy the assumptions of Theorem~\ref{thm:sub}. Define
\begin{align}
\P&=\{\bfp:\ a_{\geq 3} > R_1a_1\ \mbox{and}\ a_{\geq 3} > R_2a_2\}, \label{P}\\
S_1&=\{\sigma\in S_n \setminus \{id\}:\ a_{\geq 3} \leq R_1a_1\ \mbox{or}\ a_{\geq 3} \leq R_2a_2\},\label{S1}\\
S_2&=S_n \setminus (\{id\}\cup S_1).\label{S2}
\end{align}

\begin{lemma}\label{lemma:mostly-deg-3-moved}
For all $\sigma\in S_1$,  a.a.s.\ $\sigma\notin \aut(G_{n, \bm{d}})$.
%Fix $R_1, R_2 > 0$.  Let $\alpha_1, \alpha_2 > 0$ be such that $\alpha_1 \leq (1-\alpha_2)/(R_1 + 4)$, $\alpha_2\le 1/(R_2+4)$.  If $n_i^{\alpha_i / 2}= o(M_1/\Delta^2n^{1 - \alpha_i/2})$ for $i=1,2$, then a.a.s.\ $\sigma\notin \aut(G_{n, \bm{d}})$ for all $\sigma\in S_1$.
\end{lemma}

%If we now apply Lemma~\ref{lemma:mostly-deg-3-moved}, then letting $\mathcal{P}$ be the set of parameter vectors for which $a_{\geq 3} \leq Ra_1$ or $a_{\geq 3} \leq Ra_2$ we immediately obtain
%\begin{equation}
%\textbf{P} \left(
%\bigcup_{\bm{p} \in \mathcal{P}}
%\bigcup_{\sigma \in S_n \setminus \{id\}}
%\{H_\sigma(\G(\bfd)) \in \mathscr{H}_{\sigma, \bm{p}}\}
%\right) = o(1).
%\end{equation}

It therefore remains to bound~\eqn{sum} for $\bm{p} \in \mathcal{P}$.  Applying the union bound, it suffices to prove
\[
\sum_{\bm{p} \in \mathcal{P}}
\sum_{\sigma \in S_n\setminus\{id\}}
\sum_{H \in \mathscr{H}_{\sigma, \bm{p}}}
\pr\left( H_\sigma(\G(\bfd))=H\right) = o(1).
\]
Note that if a vertex $v$ is moved by $\sigma$, then the degrees of $v$ in $G\in\G(\bfd)$ and $H_\sigma(G)$ is always the same.  We can therefore restrict the inner sum above to those $H \in \mathscr{H}_{\sigma, \bm{p}}$ in which for all $i\in\text{supp}(\sigma)$, $d_H(i)=d_i$.  It implies the nice property that if $H \in \mathscr{H}_\sigma$, then by the definition of $\mathscr{H}_\sigma$ we know that $H = H_{\sigma}(G)$ if and only if $H \subseteq G$, for $G\in \G(\bfd)$.

Define 
\[
T(\bm{p})=\sum_{\sigma \in S_n\setminus\{id\}}
\sum_{H \in \mathscr{H}_{\sigma, \bm{p}}}
\pr\left( H_\sigma(\G(\bfd))=H\right).
\]
So the left hand side of the above equation is equal to $\sum_{\bm{p} \in \mathcal{P}}T(\bm{p})$.
We will obtain an upper bound on $T(\bm{p})$ by first calculating an upper bound on the number of pairs $(\sigma, H)$ where $H \in \mathscr{H}_{\sigma, \bm{p}}$ and the vertices moved by $\sigma$ have the same degree in $H$ as in $\G(\bfd)$.  We will then multiply this bound by a bound on $\pr(H = H_{\sigma}(\G(\bfd)))$ using the following lemma by McKay and Wormald~\cite[Corollary~3.4]{McKay}.
\begin{lemma}\label{lemma:edge-prob-bound}
Assume $\Delta^2 = o(M_1)$. There exists a constant $C$ such that given any spanning subgraph $H \subseteq K_n$ containing $q$ edges,
\[
\pr(H \subseteq G(\bfd)) \leq \left( \frac{C \Delta^2}{M_1} \right)^q.
\]
\end{lemma}

The following lemma bounds $\sum_{\bm{p} \in \mathcal{P}}
	T(\bfp) $ by $o(1)$. The proof is based on an extension of the argument by Wormald~\cite{wormald}.
	\begin{lemma}\label{lemma:other-half-of-sum}
	$
	\sum_{\bm{p} \in \mathcal{P}}
	T(\bfp) = o(1).
	$
\end{lemma}

\begin{proof}
In this proof, we use $C$ to denote an absolute constant. The various constants $C$ that appearing in the proof do not necessarily take the same value. Let $R_1$, $R_2$ and $0<\eps<1$ be constants such that all conditions of Theorem~\ref{thm:sub} are satisfied.
Given $\bm{p}$, the vertices in each of the sets $A_i$, $i\in\{1,2\}$, can be chosen in at most $\binom{n_i}{a_i}$ ways, while the vertices in $A_{\geq 3}$ can be chosen in at most $\binom{n}{a_{\geq 3}}$ ways.  Since $\sigma$ fixes $H$, it must fix each of these sets.  There are therefore at most $a_1! a_2! a_{\geq 3}! / \prod_i s_i!$ ways to choose $\sigma$ once $A_1, A_2$ and $A_{\ge 3}$ have been chosen.

The $e_1$ edges of $H_0$ which are fixed by $\sigma^*$ must correspond to 2-cycles of $\sigma$ and can therefore be chosen in at most $\binom{s_2}{e_1} \leq 2^{a_1 + a_2 + a_{\geq 3}}$ ways.  The other $k-e_1$ edges in $H_0$ must lie in $f - e_1 = k - m - e_1$ orbits, and in particular must fill these orbits.  Choosing these edges therefore amounts to choosing their corresponding orbits, and there are at most $\binom{(a_1 + a_2 + a_{\geq 3})^2}{k - m - e_1}$ ways to do so.  The graph $H$ contains $a_1 + 2a_2 + \ell - k$ edges, and therefore $a_1 + 2a_2 + \ell - 2k$ edges which are not in $H_0$.  Each of these edges joins  a vertex outside $H_0$ (which must be fixed by $\sigma$) with one inside $H_0$, and must therefore lie in an orbits of size at least 2.  The number of orbits is therefore at most $(a_1 + 2a_2 + \ell - 2k)/2$, and there are at most $n^{(a_1 + 2a_2 + \ell - 2k)/2}$ ways to choose these orbits. By the theorem assumptions,  $\Delta^2=o(M_1)$. Recalling
$
d=M_1/n
$, combining all of these bounds, and using Lemma~\ref{lemma:edge-prob-bound} it follows that
\begin{equation}\label{eq:T-product}
T(\bm{p})
< \frac{
	C^{a_1 + a_2 + \ell}
	n_1^{a_1}
	n_2^{a_2}
	n^{-a_1/2 - a_2 - \ell/2 + a_{\geq 3}}
}{
	\prod_i s_i!
}
\binom{(a_1 + a_2 + a_{\geq 3})^2}{(k - m - e_1)}
\left(\frac{\Delta^2}{d}\right)^{a_1 + 2a_2 + \ell - k}
\end{equation}
We now show that for $C_0=\frac{1}{6}-\frac{1}{2R_1}-\frac{1}{R_2}>0$,
\begin{equation}\label{eq:T-bound}
T(\bm{p})
<
	C^{a_1 + a_2 + \ell}
	\left( \frac{n_1}{n} \right)^{a_1}
	\left( \frac{n_2}{n} \right)^{a_2}
	n^{-C_0 \ell}
	\left( \frac{\Delta^2}{d} \right)^{a_1 + 2a_2 + \ell}.
\end{equation}

Let $\xi>0$ be a sufficiently small constant. %$\xi \in \left(0, \frac{3/4 - 3C_0/2}{(1/R_1 + 1/R_2 + 1)}\right)$.  
Suppose $k \leq e_1 + \xi (a_1+a_2+a_{\geq 3})$.  By \eqref{eq:T-product} it follows that
\[
T(\bm{p}) <
C^{a_1+a_2+\ell}
(a_1+a_2+a_{\geq 3})^{2\xi(a_1+a_2+a_{\geq 3})}
\left(\frac{n_1}{n^{1/2}}\right)^{a_1}
\left(\frac{n_2}{n}\right)^{a_2}
n^{-\ell/2 + a_{\geq 3}}
\left(\frac{\Delta^2}{d}\right)^{a_1 + 2a_2 + \ell - k}.
\]
By~\eqn{P}, $R_1 a_1 < a_{\geq 3}$ and $R_2 a_2 < a_{\geq 3}$. We also have $a_1+a_2+a_{\geq 3}\le n$. Moreover, we know $a_{\geq 3} \leq \ell/3$. Hence, 
\[
(a_1+a_2+a_{\geq 3})^{2\xi(a_1+a_2+a_{\geq 3})}\le n^{\frac{2}{3}\xi\ell(\frac{1}{R_1}+\frac{1}{R_2}+1)}.
\]
It follows that
\[
T(\bm{p}) <
C^{a_1+a_2 + \ell}
\left( \frac{n_1}{n^{1/2}} \right)^{a_1}
\left( \frac{n_2}{n} \right)^{a_2}
n^{-\ell \left(\frac{1}{6} - \frac{2\xi}{3}\left(\frac{1}{R_1} + \frac{1}{R_2} + 1\right)\right)}
\left(\frac{\Delta^2}{d}\right)^{a_1 + 2a_2 + \ell - k},
\]
and \eqref{eq:T-bound} follows by choosing sufficiently small $\xi$.

Suppose now that $k > e_1 + \xi (a_1+a_2 + a_{\geq 3})$.
Firstly, we know that
\[
\binom{(a_1 + a_2 + a_{\geq 3})^2}{k - m - e_1} \leq \left(\frac{e(a_1 + a_2 + a_{\geq 3})^2}{k - m - e_1}\right)^{k - m - e_1}.
% \frac{(a_1 + a_2 + a_{\geq 3})^{2(k-m-e_1)}}{(k-m-e_1)!} \leq \left( \frac{e(a_1 + a_2 + a_{\geq 3})^2}{k-m-e_1} \right)^{k-m-e_1}.
\]
It can be shown analytically that $(c/x)^x$ is an increasing function of $x$ when $x < c/e$.  Since $k \leq \binom{a_1 + a_2 + a_{\ge 3}}{2}$ (which in particular means that $k - m - e_1 < (a_1 + a_2 + a_{\geq 3})^2$), if we can find some lower bound $M \leq m$ for $m$ then 
\begin{equation}\label{eq:replace-m-with-M}
\binom{(a_1 + a_2 + a_{\geq 3})^2}{k - m - e_1}
\leq
\left( \frac{e(a_1 + a_2 + a_{\geq 3})^2}{k-M-e_1} \right)^{k-M-e_1}.
\end{equation}
We now describe an approach to calculating such an $M$.  Let $e_j$ be the number of $j$-cycles of $\sigma^*$ filled by edges of $H_0$, and let $t_{i,j}$ be the number of $i$-cycles of $\sigma$ which move vertices of degree $j$.
Then
\[
k = \sum_{j\ge 1}je_j, \quad k-m = f = \sum_{j\ge 1} e_j, \quad\text{and}\quad m = \sum_{j\ge 1}(j-1)e_j.
\]
This means that for any choice of $i \geq 1$, we know that $i(k-m) - m = \sum_j(i+1-j)e_j$, which implies that
\begin{equation}\label{eq:m-lower-bound}
m \geq \frac{1}{i+1}(ik - ie_1 - (i-1)e_2 - \dots - e_i).
\end{equation}
Note next that $2ie_i$ is equal to the number of ends of edges (counting repetitions) that are in $i$-cycles of $\sigma^*$.  Moreover, for $i\in\{2,3\}$, any $i$-cycle of $\sigma^*$ which is filled by edges of $H_0$ must join two $i$-cycles of $\sigma$ or consist of `diagonal' edges of a $2i$-cycle of $\sigma$.  Here, a diagonal edge of a $2i$-cycle $\omega \in S_n$ refers to an edge connecting two vertices which are moved by $\omega$, and which is fixed by $(\omega^*)^i$.  There are at most $\sum_j ijt_{i,j}$ ends of edges which are attached to vertices moved by $i$-cycles of $\sigma$, and at most $2is_{2i}$ ends of edges which are diagonals of $2i$-cycles of $\sigma$.  Considering the ends of edges that are fixed by $\sigma^*$ or in 2-cycles of $\sigma^*$, it follows that
\begin{equation}\label{eq:edge-2-cycle-bounds}
4e_2+ 2e_1 \leq  \sum_j 2jt_{2,j} + 4s_4  ,
\end{equation}
and considering the ends of edges that are in 3-cycles of $\sigma^*$ gives that
\begin{equation}\label{eq:edge-3-cycle-bounds}
6e_3 \leq \sum_j 3jt_{3,j} + 6s_{6}.
\end{equation}
{Taking $i=3$ in \eqref{eq:m-lower-bound} we get}{Setting $i=3$ in \eqref{eq:m-lower-bound} implies that}
\begin{equation}
m\ge \frac14(3k - 3e_1 - 2e_2 - e_3).\lab{m1}
\end{equation}
By \eqref{eq:edge-2-cycle-bounds} and~\eqn{eq:edge-3-cycle-bounds},
\begin{align*}
m
&\geq \frac14(3k - 3e_1 - 2e_2 - e_3)\\
&\geq \frac14 \left( 3k - 2e_1  - \sum_{j \geq 1}jt_{2,j} - 2s_4 - \sum_{j \geq 1} \frac{j}{2}t_{3,j} - s_6\right)\\
&\ge \frac14 \left( 3k - 2e_1 - 3s_2 - \sum_{j \geq 4}jt_{2,j} - 2s_4 - \frac{3}{2}s_3 - \sum_{j \geq 4} \frac{j}{2}t_{3,j} - s_6\right),
% &\geq \frac14 \left( 3k - 2e_1 - t_{2,1} - 2t_{2,2} - 3t_{2,3} - \sum_{j \geq 4}jt_{2,j} - 2s_4 - \frac{1}{2}t_{3,1} - t_{3,2} - \frac{3}{2}t_{3,3} - \sum_{j \geq 4} \frac{j}{2}t_{3,j} - s_6\right).
\end{align*}
since $\sum_{j\le 3} jt_{2,j} \le 3s_2$ and $ \sum_{j \le 3} \frac{j}{2}t_{3,j} \le (3/2)s_3$. Let $m_j$ be the number of vertices of degree $j$ in $A_{\geq 3}$. 
Since a vertex cannot be in a 2-cycle and a 3-cycle of $\sigma$ at the same time,   $t_{2,j} \leq (m_j - 3t_{3,j})/2$ for all $j$.  Applying this to the above inequality gives
\begin{align*}
m
&\geq \frac14 \left( 3k - 2e_1 - 3s_2 - \frac{3}{2}s_3 - 2s_4  - s_6 - \frac{1}{2}\sum_{j \geq 4}j(m_j - 2t_{3,j}) \right).
\end{align*}
We also know that
\[
\ell \ge 3m_3 + 4\sum_{j \geq 4}m_j = 4a_{\geq 3} - m_3
\]
which implies that $m_3 \geq 4a_{\geq 3} - \ell$. Therefore 
\[
\sum_{j \geq 4}j(m_j - 2t_{3,j}) \leq \sum_{j \geq 4}jm_j \le \frac{1}{4}( \ell - 3m_3 ) \leq \ell - 3a_{\geq 3}
\]
which gives
\[
m\ge M := \frac14 \left( 3k - 2e_1 - 3s_2 - \frac{3}{2}s_3 - 2s_4 - s_6 - \frac{\ell}{2} + \frac{3a_{\geq 3}}{2} \right),
\]
as our lower bound for $m$.  
This gives
\[
k - M - e_1 = \frac{1}{4} \left(k - 2e_1 + 3s_2 + \frac{3}{2}s_3 + 2s_4 + s_6 +\frac{\ell}{2} - \frac{3a_{\geq 3}}{2} \right).
\]
Since $k > e_1 + \xi(a_1+a_2 +a_{\geq 3}), e_1 \leq s_2$ and $a_{\geq 3} \leq \ell/3$, it follows that $k -M - e_1 \ge\xi (a_1 + a_2 + a_{\geq 3})$, and therefore
\begin{equation}\label{eq:halve-exponent}
\left(\frac{(a_1+a_2 + a_{\geq 3})^2}{k-M-e_1}\right)^{k-M-e_1} \leq C^{a_1+a_2 + \ell}(a_1+a_2 + a_{\geq 3})^{k-M-e_1}.
\end{equation}
By Stirling's formula, we can  bound $\prod_i s_i!$ in~\eqref{eq:T-product} from below by $(1/C)^{a_1+a_2 + \ell}\prod_{i}s_i^{s_i}$.
%Using Stirling's formula we can then rewrite the $s_i!$ in the denominator of \eqref{eq:T-product} as $C^{a_1+a_2 + \ell}s_i^{s_i}$.  
Note that in general, given $r \le 1$ then there exists $C>0$ such that $q^{rx}/x^x \leq C^q$ for all $x \geq 0$ and for all $q$ (where $0^0 := 1$).   We can therefore cancel out the exponents involving the $s_i$ temrs in the following to get \begin{equation*}
\frac{(a_1+a_2 + a_{\geq 3})^{k-M-e_1}}{\prod_i s_i!} \leq C^{a_1 + a_2 + \ell}(a_1+a_2 + a_{\geq 3})^{\frac{1}{4}(k + \ell/2 - 3a_{\geq 3}/2)}\le C^{a_1 + a_2 + \ell}n^{\frac{1}{4}(k + \ell/2 - 3a_{\geq 3}/2)} .
\end{equation*}
Using this, along with the fact that $k \leq a_1/2 + a_2 + \ell/2$, it follows that
\[
T(\bfp) < C^{a_1+a_2 + \ell} \left( \frac{n_1}{n^{1/2}} \right)^{a_1} \left(\frac{n_2}{n}\right)^{a_2} n^{-\ell/2 + a_{\geq 3}+\frac{1}{4}(k + \ell/2 - 3a_{\geq 3}/2)}.
\]
Since $a_1 < a_{\geq 3}/R_1$ and $a_2 < a_{\geq 3}/R_2$, and also $a_{\geq 3} \leq \ell/3$, this implies
\[
T(\bfp) < C^{a_1+a_2 + \ell} \left( \frac{n_1}{n^{1/2}} \right)^{a_1} \left(\frac{n_2}{n}\right)^{a_2} n^{-\ell(\frac{1}{6}-\frac{1}{2R_1}-\frac{1}{R_2})},
\]
and \eqref{eq:T-bound} follows.

It remains to bound the sum of $\eqref{eq:T-bound}$ over all $\bm{p} \in \mathcal{P}$.  Given choices of $a_1, a_2$ and $\ell$, we know that $a_{\geq 3} \leq \ell/3$ and $s_i \leq a_1 + a_2 + a_{\geq 3}$, while $k \leq a_1/2 + a_2 + \ell/2$ and $e_1$ and $m$ are both at most $k$.  There are therefore at most $(a_1 + a_2 + \ell)^C$ possibilities for all remaining parameters given values for $a_1, a_2$ and $\ell$.  Note that $\bm{p} \in \mathcal{P}$ implies that $a_{\geq 3} > 0$ and therefore that $\ell \geq 3$.  It follows that
\[
\sum_{\bm{p} \not \in \mathcal{P}} T(\bm{p})
<
\sum_{\ell \geq 3} \sum_{a_2 \geq 0} \sum_{a_1 \geq 0}
	(a_1 + a_2 + \ell)^C
	C^{a_1 + a_2 + \ell}
	\left( \frac{n_1}{n} \right)^{a_1}
	\left( \frac{n_2}{n} \right)^{a_2}
	n^{-C_0 \ell}
	\left( \frac{\Delta^2}{d} \right)^{a_1 + 2a_2 + \ell}
\]
Each sum in turn may be bounded using geometric series to obtain
\begin{align*}
\sum_{\bm{p} \not \in \mathcal{P}} T(\bm{p})
&<
\sum_{\ell \geq 3} \left( \frac{ C \Delta^2 }{n^{C_0} d} \right)^\ell
\sum_{a_2 \geq 0} \left(\frac{C n_2 \Delta^4}{n d^2}\right)^{a_2}
\sum_{a_1 \geq 0} \left(\frac{C n_1 \Delta^2}{n^{1/2} d}\right)^{a_1}\leq
\left(\frac{C \Delta^2 n^{1-C_0}}{M_1}\right)^3 = o(1),
\end{align*}
since $\frac{\Delta^2}{d}=o(n^{C_0})$ by assumption.
\end{proof}

\subsection{Proof of Lemma~\ref{lemma:mostly-deg-3-moved}}

The premise of this section is to show that, with high probability, any non-trivial automorphism has to move a large number of vertices of degree at least 3. We show that with high probability if there is a non-trivial automorphism in the random graph, then any vertex of degree 1 or 2 that is moved must lie in a component induced by moved vertices that are mostly of higher degrees. 

Recall that $V_i$ is the set of vertices  of degree $i$ and $V_{\geq 3}$ denotes the set of vertices of degree at least 3.
Given a subgraph $H$ of $G(\bfd)$, let $V_i(H) = V(H) \cap V_i$, so $V_i(H)$ is the set of vertices in $H$ which have degree $i$ in $\G(\bfd)$.  Similarly define $V_{\geq 3}{(H)}$.
Given an automorphism $\sigma$ of $\G(\bfd)$, recall that for $i \in \{1,2\}$ $a_i$ denotes the number of degree $i$ vertices moved by $\sigma$ and  $a_{\geq 3}$ denotes the number of vertices of degree at least 3 moved by $\sigma$.  Recall that $M_1$ denotes the total degree of the vertices in $\G(\bfd)$. We will first prove several probabilistic results regarding structures involving vertices with degree 1 or 2. 

Let $R_1$, $R_2$ and $\eps$ be constants that satisfy the condition of Theorem~\ref{thm:sub}. This fixes $\alpha_1$ and $\alpha_2$.

\begin{lemma}\label{lem:degree-3-between-degree-1}
Assume 
$\frac{\Delta^2}{d}=o\left(\frac{n^{1/4}}{n_1^{1/2}}\right)$ and $\frac{\Delta^2}{d}=o\left(\frac{n^{\alpha_2/2}}{n_2^{\alpha_2/2}}\right)$. A.a.s. every path joining degree 1 vertices in $G(\bfd)$ contains at least four vertices in $V_{\geq 3}$.
\end{lemma}

\begin{proof}
	Let $X_{L,j}$ denote the number of paths of length $L$, which joins a pair of vertices with degree 1, and contains exactly $j$ vertices of degree at least 3. Then, for any $L\ge 5$ and $0\le j\le 3$,
	\[
	\ex X_{L,j} \le n_1^2 \binom{n_2}{L-1-j} n^j (L-1)! \left(\frac{C\Delta^2}{M_1}\right)^L\le n_1^2 n^j n_2^{-j-1} \left(\frac{C\Delta^2 n_2}{M_1}\right)^L.
	\] 
	By assumption on $n_2$, we know $C\Delta^2 n_2=o(M_1)$. Thus,
	\[
	\sum_{0\le j\le 3}\sum_{L\ge 5} \ex X_{L,j} \le \sum_{0\le j\le 3}n_1^2 n^j n_2^{-j-1} \left(\frac{C\Delta^2 n_2}{M_1}\right)^5 \le n_1^2 n^3 n_2^{-4} \left(\frac{C\Delta^2 n_2}{M_1}\right)^5 =o(1)
	\]
	by our assumption on $n_1$ and $n_2$. The assertion of the lemma follows.
	\end{proof}
	Next, we show that every cycle must contain more than two vertices of degree at least 3.  
\begin{lemma}\label{lemma:degree-2-cycles}
	Suppose $
\frac{\Delta^2}{d}=o\left(\frac{n^{1/3}}{n_2^{1/3}}\right)$. Then a.a.s.\ there are  no cycles in $G(\bfd)$ containing two or fewer vertices in $V_{\geq 3}$.
\end{lemma}
\begin{proof}
	Let $X_K$ be the number of cycles of length $K$ containing at most 2 vertices in $V_{\geq 3}$.  Then,	\begin{align*}
	\ex X_{K}
	&\leq \sum_{j=0}^2 \binom{n_2}{K-j} \binom{n}{j} \frac{(K-1)!}{2} \left(\frac{C \Delta^2}{M_1}\right)^K\leq \frac{C^K n_2^{K-2} n^2 \Delta^{2K}}{M_1^K}.
	\end{align*}
	Summing over all $K \geq 3$ gives
	\[
	\sum_{K \geq 3}\ex X_K
	\leq \frac{n^2}{n_2^2} \sum_{K \geq 3} \left(\frac{C n_2 \Delta^2}{M_1}\right)^K
	= \frac{C n^2 n_2 \Delta^6}{M_1^3} \cdot \frac{1}{1 - o(1)}
	= o(1).
	\]
\end{proof}

We now give two lemmas about subgraphs of $G(\bfd)$ of certain sizes. These lemmas will be  combined to prove that a.a.s.\ any connected subgraph containing more edges than vertices must have a large number of  vertices in $V_{\ge 3}$. 
\begin{lemma}\label{lemma:more-edges-than-vertices}
	Let $L_1 = \frac{\ln(M_1/\Delta^2)}{\ln\left(\frac{C_1 n \Delta^2}{M_1}\right)}$, where $C_1>0$ is a sufficiently large constant.
	If $\Delta^2 = o(M_1)$, then given any $\beta_1 \in (0,1)$ there is a.a.s. no subgraph of $G(\bfd)$ on at most $\beta_1 L_1 - 1$ vertices that has more edges than vertices.
\end{lemma}
\begin{proof}
	Let $X_{K}$ be the number of subgraphs of $\G(\bfd)$ on $K$ vertices with more than $K$ edges.  %Then to find the expected number of such subgraphs, 
	There are at most $\binom{n}{K}$ ways to choose a set of $K$ vertices. Given a $K$-set $S$ of vertices, there are at most $\binom{K^2}{j}$ ways to fix $j$ edges with both ends in $S$.
 By Lemma \ref{lemma:edge-prob-bound}, 
	\begin{align*}
	\ex X_{K} &\leq \sum_{j\ge K+1} \binom{n}{K} \binom{K^2}{j} \left( \frac{C \Delta^2}{M_1} \right)^{j}\leq  \sum_{j\ge K+1} \left(\frac{n e}{K}\right)^K \left(\frac{K^2 e}{j}\right)^{j} \left( \frac{C \Delta^2}{M_1} \right)^{j}\le \frac{C^{K} n^K \Delta^{2(K+1)}}{M_1^{K+1}}.
	\end{align*}
	% If we let $Y_K$ be the number of subgraphs on $K$ vertices with more than $K$ edges, then when $K \leq \ln(n)$ it follows that
	% \begin{align*}
	% \bm{E}[Y_K]
	% &= \sum_{t \geq K+1} \bm{E}[X_{K,t}]\\
	% &\leq  \left( \frac{ne}{K} \right)^K \sum_{t \geq K+1} \left(\frac{CeK}{n} \right)^t\\
	% &= \frac{C^{K+1} e^{2K + 1}K}{n} \sum_{r \geq 0} \left( \frac{CeK}{n} \right)^r\\
	% &= \frac{C^{K+1} e^{2K+1} K}{n} \frac{1}{1 - \frac{CeK}{n}}.
	% \end{align*}
	Summing over all $K \leq \beta_1 L_1 - 1$ gives that
	\begin{align*}
	\sum_{K=1}^{\beta_1 L_1 - 1} \ex X_K &\le \frac{\Delta^2}{M_1} \sum_{K=1}^{\beta_1 L_1 - 1} \left(\frac{Cn\Delta^2}{M_1}\right)^K = \frac{\Delta^2}{M_1} \frac{((Cn\Delta^2/M_1)^{\beta_1L_1}-Cn\Delta^2/M_1)}{Cn\Delta^2/M_1-1}  \\
	&\le  \frac{\Delta^2}{M_1} \left(\frac{Cn\Delta^2}{M_1}\right)^{\beta_1L_1} \le \exp\left(\ln(M_1/\Delta^2)\left(-1+\beta_1+\frac{\beta_1C}{\ln(C_1\Delta^2n/M_1)}\right)\right)=o(1),
	% &\leq \frac{C' \Delta^2}{M_1} \sum_{K=0}^{\beta_1 L_1} \left(\frac{C' n \Delta^2}{M_1}\right)^K\\
	%&\leq O\left(\frac{\Delta^2}{M_1}\right) \cdot \frac{1 - O\left(\frac{n \Delta^2}{M_1}\right)^{\beta_1 L_1}}{1 - O\left(\frac{n \Delta^2}{M_1}\right)}\\
	%&= \frac{1}{n} \cdot O\left(\frac{n \Delta^2}{M_1}\right)^{\beta_1 L_1}\\
	%&= O\left(n^{\beta_1 - 1}\right)\\
	%&= o(1),
	\end{align*}
	by choosing sufficiently large $C_1>0$.
\end{proof}

\begin{lemma}\label{lemma:degree-2-density} 
	Fix $\alpha \in (0,1)$. Assume $\Delta^2 = o\left(\frac{M_1}{n_i^\alpha n^{1-\alpha}}\right)$ for $i=1,2$. Let  $L_2^{(i)} = \frac{\ln\left(\frac{M_1}{\Delta^2}\right)}{\ln\left(\frac{M_1}{C_2 \Delta^2 n_i^\alpha n^{1-\alpha}}\right)}$ for $i=1,2$, where $C_2>0$ is  a sufficiently large constant.
	  For any $\beta_2 > 1$, there is a.a.s.\ no connected subgraph $H \subseteq G(\bfd)$ such that $|V(H)| \geq \beta_2 L_2^{(i)}$ and $V_i(H) \geq \alpha |V(H)|$, for $i=1,2$.
\end{lemma}
\begin{proof} Let $i\in\{1,2\}$.
	Let $X_K$ be the number of connected subgraphs of $G(\bfd)$ on $K$ vertices containing at least $\alpha K$ vertices of degree $i$. First note that $\Delta^2 = o\left(\frac{M_1}{n_i^\alpha n^{1-\alpha}}\right)$ implies that $n_i=o(n)$. Hence, the assertion holds trivially if $K\ge n/4$. Next, consider $K<n/4$. 
		Let 
	\[
	f(K)= \binom{n_i}{j} \binom{n}{K - j} K^{K-2} \left( \frac{C \Delta^2}{M_1} \right)^{K-1}.
	\]
	As there are $K^{K-2}$ trees of order $K$, it follows that
	\begin{align*}
	\ex X_K\leq \sum_{j\ge \alpha K} f(K).
	\end{align*}
Since
\[
\frac{f(j+1)}{f(j)}  \le \frac{n_i}{n-K} \frac{K-j}{j} \le \frac{n_i}{3n/4}\alpha^{-1}=o(1),\quad\mbox{for all $j\ge \alpha K$,}
\]
we have
\begin{align*}
	\ex X_K&\leq  (1+o(1)) f(\alpha K) \le \left(\frac{n_ie}{\alpha K}\right)^{\alpha K} \left( \frac{ne}{K -\alpha K} \right)^{K - \alpha K} K^{K-2} \left(\frac{C \Delta^2}{M_1}\right)^{K-1}\\
	&\leq \frac{C^{K} n_i^{\alpha K} n^{K - \alpha K} \Delta^{2(K-1)}}{M_1^{K-1}}.
\end{align*}
	Summing this over all $ \beta_2 L_2^{(i)} \le K<n/4$ gives
	\begin{align*}
	\sum_{K= \beta_2 L_2^{(i)}}^{n/4}\ex X_K &\le\frac{M_1}{\Delta^2} \sum_{K \geq \beta_2 L_2^{(i)}} \left(\frac{C n_i^\alpha n^{1 - \alpha} \Delta^2}{M_1}\right)^K
	\le \frac{2M_1}{\Delta^2} \left(\frac{C n_i^\alpha n^{1 - \alpha} \Delta^2}{M_1}\right)^{\beta_2 L_2^{(i)}}\\
	& \le 2\exp((1-\beta_2)\log(M_1/\Delta^2))	% = O(1) \left(\frac{\Delta^2}{M_1}\right)^{\beta_2 - 1}
	= o(1),
	\end{align*}
	by our choice of $\beta_2$ and $L_2^{(i)}$.
\end{proof}

Combining the above two results gives us the following corollary. Given a subgraph $F\subseteq \G(\bfd)$, recall that $V_i(F)=V(F)\cap V_i$ for $i=1,2$.
\begin{cor}\label{corollary:more-edges-and-lots-of-deg-2}
	Fix $\alpha \in (0, 1)$ and $\eps>0$.  %Assume $\Delta^2 = o\left(\frac{M_1}{n_i^{\alpha/2}n^{1-\alpha/2}}\right)$ for $i=1,2$.  
	Suppose $\Delta^2 = o\left(\frac{M_1}{n_i^{\alpha/2}n^{1-\alpha/2}}\right)$ for $i=1,2$, and
	\begin{equation}
	%n^{1+(1-\alpha)(1-\eps)} \Delta^{4-2\eps} n_i^{\alpha-\alpha\eps} \ll M_1^{2-\eps},\quad \mbox{for $i=1,2$.}
	\left(\frac{n_i}{n}\right)^{\alpha(1-\eps)} \left(\frac{\Delta^2}{d}\right)^{2-\eps}=o(1) ,\quad \mbox{for $i=1,2$.} \lab{new}
	\end{equation}
Then a.a.s. every connected subgraph $F \subseteq G(\bfd)$ with more edges than vertices satisfies $|V_i(F)| < \alpha |V(F)|$.
\end{cor}
\begin{proof} Fix $i\in\{1,2\}$.
	Let $L_1$ and $L_2^{(i)}$, and $C_1$ and $C_2$ be as in Lemmas~\ref{lemma:more-edges-than-vertices} and \ref{lemma:degree-2-density}. 
	By assumption~\eqn{new}, 
	\[
	\ln\left(\frac{C_1 n\Delta^2}{M_1}\right) < (1-\eps) \ln\left(\frac{M_1}{C_2\Delta^2n_i^{\alpha}n^{1-\alpha}}\right).
	\]
	It follows immediately that there exist $0<\beta_1<1$ and $\beta_2>1$ such that $\beta_2 L_2^{(i)} < \beta_1 L_1$. By Lemmas~\ref{lemma:more-edges-than-vertices} and \ref{lemma:degree-2-density}, it is sufficient to show that any graph $G$ on $[n]$ satisfying properties in Lemmas~\ref{lemma:more-edges-than-vertices} and \ref{lemma:degree-2-density} with such $\beta_1,\beta_2$ must have the desired property in the corollary.
	Let $\mathscr{F}$ be the set of all connected subgraphs of $G$ that have more edges than vertices.  Then, every $F \in \mathscr{F}$ contains at least $\beta_1 L_1\ge \beta_2 L_2^{(i)}$ vertices.
	 It follows that  every $F \in \mathscr{F}$ satisfies $|V_i(F)| < \alpha |V(F)|$.
\end{proof}

We now prove two deterministic structural results which will be combined with the above probabilistic results to form a proof of Lemma~\ref{lemma:mostly-deg-3-moved}.

\begin{lemma}\label{lemma:no-trees-with-excess-branching}
	There does not exist a tree on 2 or more vertices in which every path joining a pair of leaves contains 2 or more vertices of degree at least 3.
\end{lemma}
\begin{proof}
	This is trivially true for trees with only two vertices.
	Say that a tree $T$ has property $Q$ if every path joining two leaves contains 2 or more vertices of degree at least 3.
	For contradiction, suppose that there exists a tree with property $Q$. For each vertex $v$ of degree two, delete $v$ from $T$ and add an edge between the two neighbours of $v$. Let $T'$ be the resulting graph (thus, $T$ is an edge subdivision of $T'$), which is a tree such that all non-leaf vertices are of degree at least 3. Moreover, no two leaves of $T'$ can have a common neighbour, since otherwise there exists a path joining two leaves in $T$ which contains only one vertex of degree at least 3. The only possible tree $T'$ is a path with length 1. Hence $T$ must be a path. But such $T$ does not have property $Q$, which contradicts with our assumption. This completes the proof.
%	If we take any degree 2 vertex in $T$ and contract both of the edges incident to it, we obtain a new graph which also has property $Q$.
%	Let $T'$ be the graph obtained by repeating this process until no degree 2 vertices remain.
	
%	Since $T'$ has property $Q$ and has no degree 2 vertices, we may assume that every leaf in $T'$ is adjacent to a distinct vertex of degree at least 3.  If $\ell$ is the number of leaves in $T'$ and $b$ is the number of other `branching' vertices (i.e. vertices of degree at least 3), then it follows that $\ell \leq b$.  However,
%	\[
%	|E(T)| = \ell + b - 1 = \frac{1}{2}\sum_{v \in V(T)}\deg(v) \geq \frac{1}{2}(\ell + 3b),
%	\]
%	which implies that $\ell > b$, a contradiction.  Therefore there cannot exist a tree with property $Q$.
\end{proof}

  In the following lemma and its proof, $V_{\geq 3}$ and $V_i$ ($i \in \{1,2\}$) refer to the sets of vertices of degree at least 3, or of degree $i$, $i \in \left\{1,2\right\}$, in a general graph $G$ rather than $\G(\bfd)$.
\begin{lemma}\label{lemma:components}
	Suppose $G$ is a graph satisfying the following conditions:
	\begin{enumerate}
		\item[(a)] \label{assumption:cycles} Every cycle $C$ in $G$ satisfies $|C \cap V_{\geq 3}| \geq 3$,
		\item[(b)] \label{assumption:paths-bw-leaves} Every path $P$ joining a pair of vertices in $V_1$ satisfies $|P \cap V_{\geq 3}| \geq 4$.
	\end{enumerate}
	Then given any $\sigma \in \text{Aut}(G)$, every vertex in $V_1 \cup V_2$ which is moved by $\sigma$ lies in a connected subgraph $H$ of $G$ such that:
	\begin{itemize}
		\item $H$ contains at most three vertices in $V_{\geq 3}$ which are fixed by $\sigma$, and
		\item $H$ has more edges than vertices.
	\end{itemize}
\end{lemma}
\begin{proof}
	Suppose $v \in V_1 \cup V_2$ is a vertex which is moved by some $\sigma \in \text{Aut}(G)$.  Let $H_v$ be the edge- and vertex-maximal connected subgraph of $G$ containing $v$ such that $H_v$ contains no vertices in $V_1 \cup V_{\geq 3}$ which are fixed by $\sigma$.
	Equivalently, construct $H_v$ by starting at $v$ and first performing a depth-first search of $G$, stopping immediately before any vertex in $\text{Fix}(\sigma)$ \emph{unless} that vertex is in $V_2$.  The subgraph $H_v$ is then the graph induced by all vertices visited by this search.
	% Alternatively, if $v \in V_2$, then let $H_v$ be the edge- and vertex-maximal connected subgraph of $G$ in which every vertex in $V_{\geq 3}$ is moved by $\sigma$. Let $P$ be the maximal path in $G$ containing $v$ and consisting only of vertices in $V_2$.  We then define $H_v$ to be the edge- and vertex-maximal connected subgraph of $G$ containing $P$ in which every vertex is moved by $\sigma$, except possibly one vertex in the centre of $P$ (this can only happen if $P$ is flipped end-to-end by $\sigma$).
	Observe that $H_v$ is isomorphic to $H_{\sigma(v)}$, and also that $H_v$ and $H_{\sigma(v)}$ must either be the same or vertex-disjoint.  Note also that if $X$ is the size of the edge cut ${\cal X}$ induced by $V(H_v)$ in $G$, then
	\begin{equation}\label{eq:subgraph-edges}
	|E(H_v)| = \frac{1}{2}\left(\sum_{u \in V(H_v)} \deg_{G}(u) - X\right),
	\end{equation}
	where $\deg_G(u)$ denotes the degree of $u$ in $G$, rather than in $H_v$.

	Suppose firstly that $X \geq 3$.  Let $v_1, v_2$ and $v_3$ be vertices in $H_v$ which are incident to three distinct edges in the ${\mathcal X}$ ($v_1, v_2$ and $v_3$ are not necessarily distinct).  For $i \in \{1,2,3\}$, the vertex $v_i$ is adjacent in $G$ to some vertex $w_i$ in $G - H_v$ which is fixed by $\sigma$, and since $\sigma$ is an automorphism it follows that $w_i$ must also be adjacent to $\sigma(v_i)$.   Since $H_v$ is connected, $H_{\sigma(v)}$ is also connected and contains $\sigma(v_1), \sigma(v_2)$ and $\sigma(v_3)$.  The graph $H$ induced by $H_v \cup \{w_1, w_2, w_3\} \cup H_{\sigma(v)}$ is therefore connected and contains at most three vertixes fixed by $\sigma$.  Moreover, it is easy to show that $H$ contains at least two cycles, and therefore strictly more edges than vertices as required. For instance, if $w_1,w_2$ and $w_3$ are distinct, then there is an $(w_i,w_j)$-path in $H_v$ and an $(w_i,w_j)$-path in $H_{\sigma(v)}$, implying a cycle passing $w_i$ and $w_j$, for each $i\neq j$. The other cases can be argued easily in a similar manner.

	% Consider the vertices adjacent to the two ends of $P$ in $G$.  By assumption \ref{assumption:cycles} we know that there is a different vertex adjacent to each end of $P$ and that at least one of these is moved by $\sigma$, since otherwise $G$ would contain a cycle through $P$ and $\sigma(P)$ containing only one or two vertices in $V_{\geq 3}$.  Furthermore, we may assume that one of these is moved and is \emph{also} in $V_{\geq 3}$ --- otherwise, if $v$ is the only moved vertex and $v$ has degree 1, there would be a path from $v$ to $\sigma(v)$ via $P$ and $\sigma(P)$ containing only one vertex in $V_{\geq 3}$, contradicting assumption \ref{assumption:paths-bw-leaves}.
	
	% Suppose now that the cut-set associated with $V(H_v)$ contains 1 or 2 edges.  We first show that $H_v$ must contain a cycle.
	% We know that $|V_{\geq 3}(H_v)| \geq 1$, and therefore if $|V_1(H_v)| = 0$ then by \eqref{eq:subgraph-edges} it follows that
	% \[
	% |E(H_v)| \geq (2(|V(H_v)| - 1) + 3 - 2)/2 = |V(H_v)| - 1/2,
	% \]
	% which implies that $|E(H_v)| \geq |V(H_v)|$ and therefore that $H_v$ contains a cycle.
	Suppose now that $X \in \{1,2\}$, and that $|E(H_v)| \geq |V(H_v)|$ (we will show that this is always the case when $X \in \{1,2\}$).  It follows that $|E(H_{\sigma(v)})| \geq |V(H_{\sigma(v)})|$.  If $H_v \cap H_{\sigma(v)} = \emptyset$, then there must be a vertex $w$ which is fixed by $\sigma$ and is incident to one vertex in $H_v$ and one in $H_{\sigma(v)}$. Let $H$ be the subgraph of $G$ induced by $H_v \cup \{w\} \cup H_{\sigma(v)}$. Then $|E(H)| > |V(H)|$ and at most one vertex in $V_{\geq 3}(H)$ fixed by $\sigma$, as required.
	On the other hand, if $H_v = H_{\sigma(v)}$ then $X=2$, since the image of any edge in ${\mathcal X}$ must also be in ${\mathcal X}$.  Moreover, both of the edges in ${\mathcal X}$ must be incident to the same vertex outside of $H_v$, call it $w$.  Let $H$ be the graph induced by $H_v \cup \{w\}$. Then $|E(H)| > |V(H)|$ again and there is at most one vertex in $V_{\geq 3}(H)$ which is fixed by $\sigma$.  We now consider three separate cases to show that if $X \in \{1,2\}$ then $|E(H_v)| \geq |V(H_v)|$.

	Suppose firstly that $V_1(H_v) = \emptyset$, which means that we must have started with $v \in V_2$.  Let $P$ be the maximal path in $G$ containing $v$ and consisting solely of vertices in $V_2$.  We claim that of the two vertices adjacent to the ends of $P$ in $G$, at least one is moved by $\sigma$ and has degree at least 3.  If neither of these vertices were moved, then they would also have to be adjacent to the ends of $\sigma(P)$.  We would then have a cycle through $P$ and $\sigma(P)$ containing only two vertices in $V_{\geq 3}$, contradicting with Lemma~\ref{lemma:components}(a).  Suppose then that only one of these is moved but has degree 1, and call it $u$.  It follows that there must be a $(u, \sigma(u))$-path containing only one vertex in $V_{\geq 3}$ (namely the vertex attached to the other end of $P$), contradicting assumption (b).  Similarly, if both are moved but have degree 1, then there is a path between them (via $P$) containing no vertices in $V_{\geq 3}$, again contradicting Lemma~\ref{lemma:components}(b).  Thus one of the vertices adjacent to the ends of $P$ has degree at least 3 and is moved by $\sigma$.  In particular, this means that $H_v$ contains at least one vertex in $V_{\geq 3}$, and using \eqref{eq:subgraph-edges} we obtain
	\[
	|E(H_v)| \geq \frac{1}{2}(2(|V(H_v)| - 1) + 3 - 2) = |V(H_v)| - \frac{1}{2},
	\]
	which implies that $|E(H_v)| \geq |V(H_v)|$.
	
	Suppose now that $X \in \{1,2\}$ but $|V_1(H_v)| = 1$, and let $v_1 \in V_1(H_v)$.  If $H_v = H_{\sigma(v)}$, then $\sigma(v_1) \in H_v$, but $\sigma(v_1) \neq v_1$ and therefore $|V_1(H_v)| \geq 2$, a contradiction.  We may therefore assume that $H_v \cap H_{\sigma(v)} = \emptyset$.  Suppose for contradiction that we also have $|V_{\geq 3}(H_v)| \leq 1$.  Letting $w$ be a vertex adjacent to vertices in $H_v$ and $H_{\sigma(v)}$ there must be a $(v_1, \sigma(v_1))$-path in $G$ which uses at most three vertices in $V_{\geq 3}$ (namely $w$, and possibly $v_3$ and $\sigma(v_3)$ if $V_{\geq 3}(H_v) = \{v_3\}$).  This contradicts with Lemma~\ref{lemma:components}(b), and we may therefore assume that $|V_{\geq 3}(H_v)| \geq 2$. From \eqref{eq:subgraph-edges} it follows that
	\[
	|E(H_v)| \geq \frac{1}{2}\left(1 + 2(|V(H_v)| - 3) + 2 \cdot 3 - 2\right) = |V(H_v)| - \frac{1}{2},
	\]
	and therefore $|E(H_v)| \geq |V(H_v)|$.
	
	Suppose finally that $X \in \{1,2\}$ and $|V_1(H_v)| \geq 2$, and suppose for contradiction that $H_v$ is a tree.  Given a tree $T$ containing three or more leaves and a leaf vertex $\ell$ in $T$, define the `branch' associated with $\ell$ to be everything along the path joining $\ell$ to the nearest vertex of degree at least 3 in $T$, excluding that vertex.  If $H_v$ contains a leaf vertex $\ell \not \in V_1$, then $\ell$ must be incident to an edge in the cut-set associated with $V(H_v)$.  For such a leaf (there can be at most one since $X\le 2$) delete the branch associated with $\ell$ from $H_v$ to obtain a graph $H_v'$.  The graph $H_v'$ must also be a tree, and must still contain at least two leaves since $|V_1(H_v')| = |V_1(H_v)| \geq 2$.  Note that every vertex in $H_v'$, except for at most two (since $X\le 2$), has the same degree in $H_v'$ as in $G$, and every leaf in $H_v'$ is in $V_1$.  By Lemma~\ref{lemma:components}(b), every path joining leaves in $H_v$ therefore contains at least four vertices in $V_{\geq 3}$, and every path joining leaves in $H_v'$ contains at least two vertices of degree at least 3.  However, by Lemma~\ref{lemma:no-trees-with-excess-branching} such a tree does not exist.  This therefore contradicts our assumption that $H_v$ is a tree, and it follows that $|E(H_v)| \geq |V(H_v)|$.

	Lastly, suppose that $X = 0$, so $H_v$ is a connected component of $G$.  We show that $H_v$ must contain two cycles, and therefore $|E(H)| > |V(H)|$.  Since $H_v$ is an isolated component, the degree of every vertex in $H_v$ is the same as its degree in $G$.  By Lemma~\ref{lemma:components}(b) and Lemma~\ref{lemma:no-trees-with-excess-branching} we therefore know that $H_v$ cannot be a tree.  Suppose then that $H_v$ contains only one cycle, so $|E(H_v)| = |V(H_v)|$.  By Lemma~\ref{lemma:components}(a) this means that $|V_{\geq 3}(H_v)| \geq 3$, which using \eqref{eq:subgraph-edges} implies $|V_1(H_v)| \geq 3$.
	
	Now, let $xy$ be an edge in the cycle in $H_v$.  If $x$ is a leaf in $H_v - xy$, then delete the branch associated with $x$ from $H_v - xy$.  Next, do the same for $y$ if $y$ is a leaf, and call the resulting tree $H_v'$.  If neither $x$ nor $y$ are leaves in $H_v$, then simply let $H_v' = H_v - xy$.  Every vertex in $H_v'$ except for at most two has the same degree in $H_v'$ as in $G$, and all of the leaves in $H_v'$ are in $V_1$.  By Lemma~\ref{lemma:components}(b), it follows that every path joining leaves in $H_v'$ contains at least two vertices of degree at least 3.  By Lemma~\ref{lemma:no-trees-with-excess-branching}, this contradicts the fact that $H_v'$ is a tree, which means that $H_v$ must actually contain at least two cycles.  Since every vertex in $H_v$ is moved by $\sigma$, letting $H = H_v$ we are done.
\end{proof}

The next two lemmas use the previous probabilistic and deterministic results to show that any non-trivial automorphism must move a large number of vertices of degree at least 3. %The proof of Theorem \ref{thm:sub} follows partially from these results. 

\begin{lemma}\label{lemma:more-deg-3-moved-than-deg-2}
	Suppose 
	\begin{align*}
	&
\frac{\Delta^2}{d}=o\left(\frac{n^{1/4}}{n_1^{1/2}}\right), \quad \frac{\Delta^2}{d}=o\left(\frac{n^{\alpha_2/2}}{n_2^{\alpha_2/2}}\right), \quad \left(\frac{n_2}{n}\right)^{\alpha_2(1-\eps)} \left(\frac{\Delta^2}{d}\right)^{2-\eps}=o(1). 
	\end{align*}
	Then a.a.s.\ $\G(\bfd)$ has no non-trivial automorphisms $\sigma$ for which $a_{\geq 3} \leq R_{2}a_2$.
\end{lemma}
\begin{proof}
	By Lemma~\ref{lemma:degree-2-cycles} (note that the assumption of the lemma is satisfied as $\alpha_2<1/4<2/3$ by its definition) and Lemma~\ref{lem:degree-3-between-degree-1}, $\G(\bfd)$ a.a.s.\ satisfies the conditions of Lemma~\ref{lemma:components}. 	Suppose $\sigma \in \text{Aut}(\G(\bfd))$, and for each $v \in V_2 \cap \text{supp}(\sigma)$ consider the connected subgraph of $\G(\bfd)$ containing $v$ as described in Lemma~\ref{lemma:components}.  Starting with this subgraph, delete all degree 1 vertices to obtain a new subgraph $F_v$, and note that $F_v$ still satisfies $|E(F_v)| > |V(F_v)|$ and $|V_{\geq 3}(F_v) \cap \text{Fix}(\sigma)| \leq 3$.
	Let $F$ be the subgraph of $\G(\bfd)$ produced by taking the union of the graphs $F_v$ over all $v \in V_2 \cap \text{supp}(\sigma)$.
	% Given $H \subseteq G_{n, \bm{d}}$, let $a_2(H)$ and $a_{\geq 3}(H)$ be the number of vertices in $V_2(H)$ and $V_{\geq 3}(H)$ respectively which are moved by $\sigma$.
	Since each $F_v$ contains more edges than vertices, by Corollary~\ref{corollary:more-edges-and-lots-of-deg-2}, a.a.s. for every component $C_i$ of $F$ we have $|V_2(C_i)| < \alpha_2 |V(C_i)|$.
	Then
	\begin{align*}
	a_{\geq 3}(C_i)
	&= |V(C_i)| - |V_2(C_i)| - |V_{\geq 3}(C_i) \cap \text{Fix}(\sigma)|\geq |V(C_i)| - |V_2(C_i)| - 3|V_2(C_i)|\\
	&> |V(C_i)| - \frac{4}{4+R_2}|V(C_i)|> R_2 |V_2(C_i)|.
	\end{align*}
	This applies to every component $C_i$ of $F$.  Since every degree 2 vertex moved by $\sigma$ lies in $F$, it follows that $a_{\geq 3} > R_2a_2$.\\
\end{proof}

\begin{lemma}\label{lemma:more-deg-3-moved-than-deg-1}
	Suppose 
	\begin{align*}
	&
\frac{\Delta^2}{d}=o\left(\frac{n^{1/4}}{n_1^{1/2}}\right), \quad \frac{\Delta^2}{d}=o\left(\frac{n^{\alpha_2/2}}{n_2^{\alpha_2/2}}\right), \quad  \left(\frac{n_i}{n}\right)^{\alpha_i(1-\eps)} \left(\frac{\Delta^2}{d}\right)^{2-\eps}=o(1),\ \mbox{for $i=1,2$.} 
	\end{align*}
	Then a.a.s.\ $\G(\bfd)$ has no automorphisms for which $a_{\geq 3} \leq R_1a_1$.
\end{lemma}
\begin{proof}
	Again, $\G(\bfd)$ a.a.s. satisfies the conditions of Lemma~\ref{lemma:components}.  Suppose $\sigma \in \text{Aut}(\G(\bfd))$ is non-trivial and for each $v \in V_1 \cap \text{supp}(\sigma)$ let $F_v$ be the connected subgraph described in Lemma~\ref{lemma:components}.  Let $F$ be the union of the graphs $F_v$ over all such $v$, we obtain a subgraph of $\G(\bfd)$ in which every component is the graph $F_v$ for some $v \in V_1 \cap \text{supp}(\sigma)$.  Each component therefore contains more edges than vertices, and hence  a.a.s., for each component $C_i$,   $|V_2(C_i)| < \alpha_2|V(C_i)|$ and $|V_1(C_i)| < \alpha_1 |V(C_i)|$.  Then
	\begin{align*}
	a_{\geq 3}(C_i)
	&= |V(C_i)| - |V_1(C_i)| - |V_2(C_i)| - |V_{\geq 3}(C_i) \cap \text{Fix}(\sigma)|\\
	&> (1 - \alpha_1 - \alpha_2 - 3\alpha_1)|V(C_i)| \geq R_1 \alpha_1 |V(C_i)|> R_1|V_1(C_i)|
	\end{align*}
	Since every vertex in $V_1$ which is moved by $\sigma$ lies in $F$, it follows that $a_{\geq 3} > R_1 a_1$.\\
\end{proof}

\begin{proof}[Proof of Lemma \ref{lemma:mostly-deg-3-moved}] This follows by Lemmas~\ref{lemma:more-deg-3-moved-than-deg-2} and~\ref{lemma:more-deg-3-moved-than-deg-1}. 
%	It remains to be shown that the assumptions of Lemmas \ref{lem:deg1path}--\ref{lemma:more-deg-3-moved-than-deg-1} are implied by the conditions of Theorem \ref{thm:super}. Lemma \ref{lem:deg1path} only requires that $n_1 = o(n^{1/2})$ and $\Delta^2 = o(M_1)$. Lemma \ref{lemma:degree-2-between-degree-1} is satisfied by taking $\alpha = \frac{\alpha_2}{2}$. Since $\alpha_2 < \frac{1}{4}$ and $n >> n_2$, this also implies that Lemma \ref{lemma:degree-2-cycles} is satisfied. Lemma \ref{lemma:more-edges-than-vertices} only requires that $\Delta^2 = o(M_1)$, which is true. Lemma \ref{lemma:degree-2-density} is true by setting $\alpha = \frac{\alpha_j}{2}$ for $i \in \left\{1,2\right\}$ as required, as is Corollary \ref{corollary:more-edges-and-lots-of-deg-2}. Finally, Lemmas \ref{lemma:more-deg-3-moved-than-deg-2} and \ref{lemma:more-deg-3-moved-than-deg-1} use the exact assumptions of Theorem \ref{thm:sub}, and thus are true by default. Thus $G_{n,\bm{d}}$ a.a.s. has no automorphisms where $a_{\geq 3} \leq R_1 a_1$ or or $a_{\geq 3} \leq R_2 a_2$.
\end{proof}
\begin{proof}[Proof of Theorem \ref{thm:sub}] It follows immediately from Lemmas \ref{lemma:mostly-deg-3-moved} and \ref{lemma:other-half-of-sum} by taking the union bound.  
\end{proof}

\section{Symmetry: proof of Theorem~\ref{thm:super}}
\label{sec:super}

In this section we  prove that under the conditions of Theorem~\ref{thm:super}, a.a.s.\ $ \G(\bfd)$  has non-trivial automorphisms with a probability bounded away from zero. We will look at automorphisms that fix all but  two vertices. In particular, we will prove the existence (with nonzero probability) of such automorphisms by proving the existence of  a cherry or a pendant triangle, which we define below. 

We say that a triple of vertices $u,v,w$ forms a {\em cherry} if the degrees of $u$ and $v$ are 1, and they are both adjacent to $w$. We say that a triple of vertices $u,v,w$ forms a {\em pendant triangle} if they are pairwise adjacent and at least two of them have degree 2. Obviously, if $G$ contains a cherry or a pendant triangle $u,v,w$ then $\sigma=(u,v)$ is an automorphism. 

We will prove that under the conditions of Theorem~\ref{thm:super}, $\G(\bfd)$ has a cherry or a pendant triangle with a non-vanishing probability. To do this, we use the following technical lemma for the edge probability in $\G(\bfd)$ conditional on the presence of a constant number of edges. Its proof is exactly the same as~\cite[Lemma 3]{Gao2018}, although the conditions for $\bfd$ in~\cite[Lemma 3]{Gao2018} are different from that in Theorem~\ref{thm:super}. We omit the proof.

 % We look at small automorphisms that fix all but a few vertices. These appear to be the first automorphisms to appear, which we do not prove. Switchings give the probabilities of these structures appearing and some basic probabilistic bounds to show that above the threshold given in Theorem~\ref{thm:super} you see automorphisms with nonzero probability. First we give the switching lemma.

\begin{lemma}\label{lem:switch}	
	Given set $C\subseteq E(K_n)$, let $E_C$ be the event that $C \subseteq E(\G(\bfd))$. Let $C_i \subseteq C$ be the set of edges in $C$ incident to vertex $i$. Assume $uv \notin C$. Provided that $|C| = O(1)$ and $\Delta^2 = o(M_1)$, we have
	\begin{align}\label{eqn:edgeprob}
		\mathbb{P} \left( u \sim v | E_C \right) = (1 + o(1)) \frac{\left(d_u - |C_u|\right)\left(d_v - |C_v|\right)}{M_1 + \left(d_u - |C_u|\right)}.
	\end{align}
	
\end{lemma}

Now we prove Theorem~\ref{thm:super}. 
%Theorem \ref{thm:sub} implies that for automorphisms to be present, an excess of degree 1 or 2 vertices is needed. Under the assumptions of Theorem \ref{thm:super}, small isolated automorphisms, such as a swap of two degree 1 vertices at distance 2, occur with nonvanishing probability. To examine this,
Let $Y$ be the number of cherries in $\G(\bfd)$, and let $Z$ denote the number of pendant triangles in $\G(\bfd)$.
First consider the case that $M_2=o(M_1)$. This implies that 
\[
\sum_{i\in V_{\ge 2}} d_i^2 = o\left(\sum_{i\in V_{\ge 2}} d_i +n_1\right),
\]
which yields $n_1=(1-o(1))n$. This immediately implies that either $Y>0$ or there exists an edge joining to vertices of degree 1, both of which implies the symmetry of $\G(\bfd)$.

Next, we assume $M_2=\Omega(M_1)$. We will show that under the assumptions of either (a) or (b), the probability that $Y$ or $Z$ being 0 is away from 1. It is easy  to see that under the assumptions of $\Delta^2=o(M_1)$ and $M_2=\Omega(M_1)$, 
\[
M_4=o(M_2 M_1)=o(M_2^2), \quad M_3=O(M_2 \Delta)=o(M_2 \sqrt{M_1})=o(M_2^{3/2}).
\]
We first estimate $Y$.
There are $\binom{n_1}{2}$ ways to choose the two vertices of degree 1. Hence, by Lemma~\ref{lem:switch}, 
\begin{align*}
	\mathbb{E}(Y) &=  {n_1 \choose 2} \sum_{w\in [n]\setminus V_1} (1+o(1))\frac{d_w(d_w-1)}{M_1^2} \sim \frac{n_1^2}{2} \frac{M_2}{M_1^2},
\end{align*}
by noting that $\sum_{w\in V_1} d_w(d_w-1)=0$ and $n_1=\Omega(M_1/\sqrt{M_2})=\omega(1)$, as $\Delta^2=o(M_1)$ by our assumption.
Next we estimate the second moment. Let $A$ be a 3-subset of $[n]$ and let $Y_{A}$ be the indicator variable for the event that $A$ forms a cherry. Then,
\[
\ex Y(Y-1)=\sum_{(A_1,A_2)}\pr(Y_{A_1} = Y_{A_2} = 1),
\]
where the summation is over all ordered distinct pairs of 3-subsets $(A_1,A_2)$.

We split the sum according to the size of  $A_1\cap A_2$.

%\begin{align}
%	\mathbb{E}(Y(Y-1)) &= \sum_{|A_1 \cap A_2| = 0 } \mathbb{P}(Y_{A_1} = Y_{A_2} = 1) \label{eqn-cherry-disjoint}\\
%	& + \sum_{|A_1 \cap A_2| = 1} \mathbb{P}(Y_{A_1} = Y_{A_2} = 1) \label{eqn-cherry-int1}\\
%	& + \sum_{|A_1 \cap A_2| = 2} \mathbb{P}(Y_{i_1} = Y_{i_2} = 1) \label{eqn-cherry-int2}.
%\end{align}
{\em Case 1:} $A_1\cap A_2=\emptyset$. In this case, there are $\binom{n_1}{2}\binom{n_2-2}{2}$ ways to choose the vertices of degree 1 for $A_1$ and $A_2$. Then, summing over distinct pairs of vertices in $[n]\setminus V_1$ gives

%then there are two subcases for when the structures intersect. If they intersect on just one vertex, they must intersect on $w$. If they intersect on two vertices, they must intersect on $w$ and a degree 1 vertex $v$ (equivalently $u$ since the sets are unordered). 

%By Lemma \ref{lem:switch}, the probability in the summation of (\ref{eqn-cherry-disjoint}) is asymptotically given by 
\begin{align*}
\sum_{(A_1,A_2): A_1\cap A_2=\emptyset}\pr(Y_{A_1} = Y_{A_2} = 1)&\sim {n_1\choose 2}{n_1-2\choose 2} \sum_{\substack{w_1,w_2\in [n]\setminus V_1\\ w_1\neq w_2}}  \frac{d_{w_1}\left(d_{w_1}-1\right)d_{w_2}(d_{w_2}-1)}{M_1^4}\\
	 &\sim \frac{n_1^4}{4M_1^4}(M_2^2+O(M_4+M_1)) \sim \frac{n_1^4M_2^2}{4M_1^4},
	\end{align*}
	noting that $\sum_{w\in[n]} d_w^2(d_w-1)^2=O(M_4+M_1)$.

{\em Case 2:} $A_1\cap A_2\neq\emptyset$. Let $A_1=\{u_1,v_1,w_1\}$ and $A_2=\{u_2,v_2,w_2\}$ where the degrees of $w_1$ and $w_2$ are at least 2. In this case, we must have  $w_1=w_2$. Hence, $A_1\cup A_2$ must induce a star with 3 or 4 leaves, all of which must be vertices in $V_1$. Given a star $S$ of order $j=4,5$, there is $O(1)$ ways choose $A_1$ and $A_2$ such that $A_1\cup A_2$ induces a subgraph isomorphic to $S$.  Hence, by the theorem assumptions,
\begin{align*}
	\sum_{(A_1,A_2): A_1\cap A_2\neq \emptyset}\pr(Y_{A_1} = Y_{A_2} = 1)&=O(1) \sum_{j=3}^4 n_1^j \sum_{w\in [n]\setminus V_1} \frac{(d_w)_j}{M_1^j} = O\left(\frac{n_1^3M_3}{M_1^3}+\frac{n_1^4 M_4}{M_1^4}\right)=o\left(\frac{n_1^4M_2^2}{M_1^4}\right),
	\end{align*}
	where the error $n_1^3M_3/M_1^3$ is absorbed since $M_3=o(M_2^{3/2})$ as we have shown, and $n_1=\Omega(M_1/\sqrt{M_2})$ by theorem assumptions.
It follows then that
\[
\ex Y(Y-1) =(1+o(1))\frac{n_1^4M_2^2}{4M_1^4}.
\]	
If $n_1=\omega(M_1/\sqrt{M_2})$, then $\ex Y\to \infty$ and $\ex Y^2\sim (\ex Y)^2$. Hence, $Y > 0$ a.a.s.\ by Chebyshev's inequality. This confirms the first case of part (a).

If $n_1=cM_1/\sqrt{M_2}$ where $c=O(1)$ then $\ex Y\sim c^2/2$, and $\ex Y^2 \sim c^4/4+c^2/2$. By the Payley--Zygmund inequality,
\[
\pr(Y>0)\ge \frac{(\ex Y)^2}{\ex Y^2} =(1+o(1))\frac{c^2}{2+c^2}.
\]
As $Y>0$ implies that $\G(\bfd)$ is symmetric, it follows that $\pr(\G(\bfd) \ \mbox{symmetric})\ge \pr(Y>0)\ge (1+o(1))\frac{c^2}{2+c^2}$. This confirms the first case of part (b). 

 Using similar arguments as for $Y$, it is easy to see that

\[
\ex Z\sim {n_2 \choose 2} \sum_{z \in [n]} \frac{4d_z(d_z-1)}{M_1^3}\sim 2n_2^2 \frac{M_2}{M_1^3},
\]
and
\begin{align*}
\ex Z(Z-1) \sim & {n_2 \choose 2}\binom{n_2-2}{2} \sum_{\substack{z_1,z_2 \in [n]\\ z_1\neq z_2}} \frac{16d_{z_1}(d_{z_1}-1) d_{z_2}(d_{z_2}-1)}{M_1^6}\\
& + O(1) n_2^4 \sum_{z\in [n]} \frac{(d_z)_4}{M_1^6} \sim \frac{4n_2^4(M_2^2+O(M_4+M_1))}{M_1^6}+O(n_2^4M_4/M_1^6).
\end{align*}
The rest of the proof is the same as for $Y$. We omit the details. \qed

\section*{Appendix}

Here we prove that for $\bfd$ from Example \ref{example-gap}, $\G(\bfd)$ is a.a.s.\ symmetric and connected.

\noindent {\em Proof}. It is easy to see that $\Delta^2=o(M_1)$, $M_1M_3=o(n_1M_2^2)$, $M_4=o(M_2^2)$ and $n_1=\omega(M_1/\sqrt{M_2})$. Hence, by Theorem~\ref{thm:super}(a), $\G(\bfd)$ is a.a.s.\ symmetric.  
It only remains to show that $\G(\bfd)$ is a.a.s.\ connected.

We will use the following theorem by Luczak~\cite{Luczak}. 

\begin{thm}\label{thm-luczak}
	If $ 3 \leq \delta = d_1 \leq \dots \leq d_n = \Delta = n^{0.02}$ then a.a.s.\ $\G(\bfd)$ is $\delta$-connected.
\end{thm}

Note that $\G(\bfd)$ is a.a.s.\ connected is an immediate consequence of the following claim.
\begin{claim}
\begin{itemize}
\item[(i)] No two vertices of degree 1 are adjacent;
\item[(ii)] No vertex is adjacent to more than 2 vertices of degree 1;
\item[(iii)] The subgraph induced by the set of vertices with degree $\ceil{\log n}$ is connected.
\end{itemize}
\end{claim}

%We cannot apply this theorem directly, so we will define a graph $G^\prime$ to be the graph obtained by deleting all degree 1 vertices from $G$. We will apply the theorem of Luczak to the graph $G^\prime$ to show that it is connected. Finally, we will show that degree 1 vertices in $G$ almost surely do not form adjacent pairs, so $G^\prime$ will be connected if and only if $G$ is connected. 

First we prove part (i). By Lemma~\ref{lem:switch}, the expected number of adjacent pairs of degree 1 vertices is 
\begin{align*}
	{n_1 \choose 2} \pr(u_1 \sim u_2) &\sim \frac{n_1^2}{2} \frac{1}{M_1} = o(1),
\end{align*}
as $n_1=o(\sqrt{n\log n})$ and $M_1\sim n\log n$.

Next, we prove (ii). Let $v$ be such that $d_v=\ceil{\log n}$.
Let $V_1=\{i:\ d_i=1\}$ and 
\[
X_v=|\{\{u_1,u_2,u_3\}\in V_1^3:\  \mbox{$u_1, u_2,u_3$ are distinct, and } v\sim u_i,\ \forall 1\le i\le 3\}|.
\]
 Without loss of generality, we may assume that $V_1=[n_1]$. Then, by Lemma~\ref{lem:switch},
\[
\ex X_v\sim \binom{n_1}{3} \frac{(\log n)_3}{M_1^3} =O\brk{\frac{\log^{3/2} n}{n^{3/2}}}.
\]
It follows that $\sum_{v\in [n]\setminus V_1} =O\brk{\frac{\log^{3/2} n}{n^{1/2}}}=o(1)$, which confirms part (ii).

Finally, we prove part (iii). Let $\bfU=(\bfu_1,\bfu_2,\ldots, \bfu_{n_1})$ where $\bfu_i$ denotes the neighbour of $i$ in $\G(\bfd)$ for $1\le i
\le n_1$. Note that $\bfU$ is a random vector. We call $\bfU$ \emph{nice} if there are no $i<j<k$ such that $\bfu_i=\bfu_j=\bfu_k$, and $\bfu_i\notin [n_1]$ for all $1\le i\le n_1$. By parts (i) and (ii), a.a.s.\ $\bfU$ is nice. Given any nice $U=(u_1,\ldots, u_{n_1})$, let $\bfd'=(d'_j)_{n_1+1\le j\le n}$ where 
\[
d'_j= d_j - \sum_{i=1}^{n_1} \ind{u_i=j},\quad\mbox{for all $n_1+1\le j\le n$.}
\]
Since $U$ is nice, $d'_j\ge \log n-3\ge 3$ for all $n_1+1\le j\le n$. Moreover, 
conditional on $\bfU=U$, the subgraph of $\G(\bfd)$ induced by $[n]\setminus [n_1]$ is $\G(\bfd')$. By Theorem~\ref{thm-luczak}, this subgraph is a.a.s.\ connected. This completes the proof of the claim.\qed


\begin{thebibliography}{10}

\bibitem{Bollobas}
B{\'e}la Bollob{\'a}s.
\newblock The asymptotic number of unlabelled regular graphs.
\newblock {\em Journal of the London Mathematical Society}, 2(2):201--206,
  1982.

\bibitem{Erdos59}
P.~Erd\"{o}s and A.~R\'{e}nyi.
\newblock On random graphs i.
\newblock {\em Publicationes Mathematicae Debrecen}, 6:290, 1959.

\bibitem{Erdos60}
Paul Erd{\H{o}}s and Alfr{\'e}d R{\'e}nyi.
\newblock On the evolution of random graphs.
\newblock {\em Publ. Math. Inst. Hung. Acad. Sci}, 5(1):17--60, 1960.

\bibitem{Erdos63}
Paul Erd{\H{o}}s and Alfr{\'e}d R{\'e}nyi.
\newblock Asymmetric graphs.
\newblock {\em Acta Mathematica Hungarica}, 14(3-4):295--315, 1963.

\bibitem{Federico17}
Lorenzo Federico and Remco Van~der Hofstad.
\newblock Critical window for connectivity in the configuration model.
\newblock {\em Combinatorics, Probability and Computing}, 26(5):660--680, 2017.

\bibitem{Gao2018}
Pu~Gao, Remco van~der Hofstad, Angus Southwell, and Clara Stegehuis.
\newblock Counting triangles in power-law uniform random graphs, 2018.

\bibitem{Gilbert59}
Edgar~N Gilbert.
\newblock Random graphs.
\newblock {\em The Annals of Mathematical Statistics}, 30(4):1141--1144, 1959.

\bibitem{Kim02}
Jeong~Han Kim, Benny Sudakov, and Van~H Vu.
\newblock On the asymmetry of random regular graphs and random graphs.
\newblock {\em Random Structures \& Algorithms}, 21(3-4):216--224, 2002.

\bibitem{Linial16}
Nati Linial and Jonathan Mosheiff.
\newblock On the rigidity of sparse random graphs.
\newblock {\em Journal of Graph Theory}, 85(2):466--480, 2017.

\bibitem{Luczak}
Tomasz Luczak.
\newblock Sparse random graphs with a given degree sequence.
\newblock In {\em Proceedings of the Symposium on Random Graphs, Poznan}, pages
  165--182, 1989.

\bibitem{McKay}
Brendan~D. McKay and Nicholas~C. Wormald.
\newblock Automorphisms of random graphs with specified vertices.
\newblock {\em Combinatorica}, 4(4):325--338, 1984.

\bibitem{wormald}
Nicholas~Charles Wormald.
\newblock A simpler proof of the asymptotic formula for the number of
  unlabelled $ r $-regular graphs.
\newblock {\em Indian Journal of Mathematics}, 28:43--47, 1986.

\bibitem{Wright71}
Edward~M Wright.
\newblock Graphs on unlabelled nodes with a given number of edges.
\newblock {\em Acta Mathematica}, 126(1):1--9, 1971.

\bibitem{Wright74}
Edward~M Wright.
\newblock Asymmetric and symmetric graphs.
\newblock {\em Glasgow Mathematical Journal}, 15(1):69--73, 1974.

\end{thebibliography}
\end{document}